\documentclass[journal]{IEEEtran}

\usepackage{hyperref}
\usepackage{graphicx,color}
\graphicspath{{./IMG/}}
\usepackage{amsmath}
\usepackage{amssymb}
\usepackage{mathrsfs}
\usepackage[ruled,linesnumbered]{algorithm2e}
\usepackage{subfigure}
\usepackage{url}
\usepackage{booktabs}
\usepackage{array}
\usepackage{lipsum}  
\usepackage[table]{xcolor}  		
\usepackage{mathtools} 			
\usepackage{dsfont} 					
\usepackage{amsthm}
\usepackage{enumitem}   			
\usepackage{siunitx} 					
\usepackage{afterpage} 					

\newtheorem{theorem}{Theorem} 
\newtheorem{lemma}[theorem]{Lemma}
\newtheorem{assumption}[theorem]{Assumption}

\newtheorem{remark}{Remark}
\newtheorem{example}{Example}

\newcommand{\mc}{\mathcal}
\newcommand{\Ker}{\operatorname{Ker}}
\newcommand{\Basis}{\operatorname{Basis}}

\newcommand{\trace}{\operatorname{Trace}}
\newcommand{\real}{\mathbb{R}}

\newcommand{\realpos}{\mathbb{R}_{\geq 0}}
\newcommand{\complex}{\mathbb{C}}
\newcommand{\transpose}{\mathsf{T}}  
\newcommand{\tsp}{\mathsf{T}} 
\newcommand{\inv}{{-1}} 
\newcommand{\negat}{\scalebox{0.75}[.9]{\( - \)}}

\newcommand\oprocendsymbol{\hbox{$\square$}}
\newcommand\oprocend{\relax\ifmmode\else\unskip\hfill
\fi\oprocendsymbol}

\newcommand{\map}[3]{#1: #2 \rightarrow #3}
\newcommand{\setdef}[2]{\{#1 \; : \; #2\}}
\newcommand{\subscr}[2]{{#1}_{\textup{#2}}}
\newcommand{\supscr}[2]{{#1}^{\textup{#2}}}
\newcommand{\until}[1]{\{1,\dots,#1\}}

\begin{document}
%
\title{Gramian-Based Optimization for the Analysis and Control of Traffic Networks} 
%
%
%

\author{Gianluca~Bianchin,~\IEEEmembership{Student~Member,~IEEE,}
        and~Fabio~Pasqualetti,~\IEEEmembership{Member,~IEEE}
\thanks{
	This material is based upon work supported in part by ARO award 
	71603NSYIP, and in part by NSF award CNS1646641.
The authors are with the Department of Mechanical Engineering, 
University of California, Riverside,
 \{\href{mailto:gianluca@engr.ucr.edu}{\texttt{gianluca}},
\href{mailto:fabiopas@engr.ucr.edu}{\texttt{fabiopas\}@engr.ucr.edu}}}
}

\maketitle

\begin{abstract}
This paper proposes a simplified version of classical models for urban 
transportation networks, and studies the problem of controlling intersections 
with the goal of optimizing network-wide congestion.
Differently from traditional approaches to control traffic signaling, a
simplified framework allows for a more tractable analysis of the network overall 
dynamics, and enables the design of critical parameters 
while considering network-wide measures of efficiency.
Motivated by the increasing availability of real-time high-resolution traffic 
data, we cast an optimization problem that formalizes the goal of 
minimizing the overall network congestion by optimally controlling the 
durations of green lights at intersections.
Our formulation allows us to relate congestion objectives with the problem of 
optimizing a metric of controllability of an associated dynamical network.
We then provide a technique to efficiently solve the optimization by 
parallelizing the computation among a  group of distributed agents.
Lastly, we assess the benefits of the proposed modeling and optimization 
framework  through microscopic simulations on typical traffic commute 
scenarios for the area of Manhattan.
The optimization framework proposed in this study is made available online 
on a Sumo microscopic simulator based interface \cite{gitHubCode}.
\end{abstract}


%
\IEEEpeerreviewmaketitle



\section{Introduction}
Effective control of transportation systems is at the core of 
the smart city paradigm, and has the potential for improving  
efficiency  and reliability of urban mobility.
%
Modern urban transportation architectures comprise two fundamental 
components: traffic intersections and interconnecting roads. 
Intersections connect and regulate conflicting traffic flows 
among adjacent roads, and their effective control can 
sensibly improve travel time and prevent congestion.
Congestion is the result of networks operating close to their capacity, 
and often leads to degraded throughput and increased travel time.

The increasing availability of sensors for vehicle detection and flow estimation, 
combined with modern communication capabilities (e.g. vehicle-to-vehicle 
(V2V) and vehicle-to-infrastructure (V2I) communication), have inspired the 
development of infrastructure control algorithms  that are adaptive 
\cite{varaiya2013max}, that is, policies that adjust the operation of the system 
based on the current traffic conditions. 
Nevertheless, the remarkable complexity of modern urban transportation 
infrastructures has recently promoted the diffusion of control policies that are 
distributed, that is, algorithms that adapt the operation of individual network 
components based on partial knowledge of the current network state and 
dynamics. 
The lack of a global network model, capable of capturing the interactions 
between spatially-distributed components and capable of modeling all the 
relevant network dynamics,   
often results in suboptimal performance \cite{diakaki2002multivariable}.
In this paper, we propose a simplified model to capture the time and spatial
relationships between traffic flows in urban traffic networks in regimes of 
free-flow.
The model represents a tradeoff between accuracy and tractability, and sets 
out as a tractable framework to study efficiency and reliability of this class of 
dynamical systems. 
The proposed framework is employed in this work for the control of green split 
times at the signalized intersections.

{\bf Related Work:}
The design of feedback policies for the control of urban 
infrastructures is an intensively studied topic, and the proposed 
techniques can mainly be divided into two categories: 
routing policies and intersections control.
Routing policies use a combination of turning preferences and 
speed limits in order to optimize congestion objectives, and have 
been studied both in a centralized \cite{lovisari2014stability} and 
distributed \cite{ba2015distributed} framework.
Conversely, intersection control refers to the design of the 
scheduling of the (automated) intersections so that the flow through  
intersections is  maximized, and can be achieved 
(i) by controlling the signaling sequence and offset, and/or 
(ii) by designing the durations of the signaling phases. 
The control of signals offset typically aims at tuning the synchronization of 
green lights between adjacent  intersections in order to produce 
green-wave effects \cite{gomes2015bandwidth,coogan2017offset}, 
and consists of solving a group of optimization problems that take into account 
certain subparts of the infrastructure, while minimizing metrics such as the 
number of stops experienced by the vehicles.
%
In contrast, the durations of green time at intersections impact the behavior of 
certain traffic flows within the network, and plays a significant role in the 
efficiency of large-scale dynamical networks \cite{diakaki2002multivariable}.

Widely-used distributed signaling control methods include 
SCOOT \cite{hunt1982scoot}, 
RHODES \cite{mirchandani2001real}, 
OPAC \cite{gartner2001implementation}, and emerge as the most common 
techniques currently employed in major cities.
The sub-optimal performance of the above methods 
\cite{wongpiromsarn2012distributed} has motivated the 
development of max-pressure techniques \cite{varaiya2013max}.
Max-pressure methods use a discrete-time model where 
queues at intersections have unlimited queue lengths. 
Under this  assumption, max pressure is proven to maximize the 
throughput by stabilizing the network.
%
Centralized policies require higher modeling efforts but, in general,
have better performance guarantees 
\cite{papageorgiou2003review}.
Among the centralized policies, the Traffic-Responsive Urban 
Control  framework \cite{diakaki2002multivariable} has received 
considerable  interest for its simplicity and good performance.
Based on a store-and-forward modeling paradigm, the method 
consists of optimizing network queue lengths through a 
linear-quadratic regulator problem that uses a relaxation of the 
physical constraints to abide with the high complexity.
Variations of these techniques to incorporate physical constraints 
have been studied in \cite{aboudolas2009store,de2010multi}.
The increased complexity of urban traffic networks has recently motivated the 
development of simplified (averaged) models to deal with the switching nature 
of the traffic signals \cite{canudas2017average}. 
However, the highly-nonlinear behavior of this class of dynamical systems 
still limits our capability to consider adequate optimization and prediction 
horizons \cite{grandinetti2018}, and the development of tractable models 
capable of capturing all the relevant network  dynamics is still an open 
problem.

{\bf Contribution:}
Motivated by the considerable complexity of traditional models for urban 
transportation systems, in this paper we propose a simplified framework to 
capture the behavior of traffic networks that are operating close 
to the free-flow regime.
In this model, each road is associated with multiple state-variables, 
representing the spatial evolution of traffic densities within the road.
This assumption allows us to capture the non-uniform spatial displacement of 
traffic within each road, 
and to construct a simplified network model that results in a more-%
tractable framework for optimization.

We employ the proposed model to formalize the goal of 
optimally designing the durations of green splits at intersections, and we  
provide a connection with the optimization of  a metric of controllability  for a 
dynamical model associated with the traffic network.
To the best of the authors' knowledge, this work represents a novel, 
computationally-tractable, method to perform network-wide optimization of 
the green-splits durations at intersections.
We provide conditions that guarantee network stability, and we  
characterize the performance of the system in relation to the overall 
network congestion.
We use the concept of smoothed spectral abscissa 
\cite{vanbiervliet2009smoothed} to numerically solve the 
optimization, and we demonstrate the benefits of our 
methods through of a microscopic simulator on the area of Manhattan.
We characterize the complexity of our algorithms, and propose a method to 
parallelize the computation in order to be more-efficiently executed by a group 
of distributed cooperating agents.
Our results and simulations suggest that the increased system performance 
and stability obtained by the methods justify the increment in complexity 
deriving from a network-wide model description.
%

{\bf Organization:}
The rest of this paper is organized as follows. 
Section \ref{sec: problem setup} presents the problem setup, relates the 
underlying assumptions to previously-established traffic models, 
and formalizes the goal of minimizing an overall measure of  network 
congestion by selecting the durations of the green split times at the 
intersections.
Section \ref{sec: optimal design} discusses our approach to solve 
the resulting nonlinear optimization problem.
The solution method is at first presented from a centralized perspective, and 
subsequently adapted for implementation over a distributed architecture 
(Section~\ref{sec:distributed}).
In the distributed paradigm, the optimization shall be solved by a group of 
cooperating agents, each with limited (local) knowledge of the physical 
network interconnection.
Section \ref{sec: simulations} is devoted to numerical simulations that 
validate our assumptions and the solution technique. 
Section \ref{sec: conclusions}  concludes the  paper.

\section{Dynamical Model of Traffic Networks and Problem Formulation}
\label{sec: problem setup}
We model urban traffic networks as a group of one-way roads interconnected 
through signalized intersections. 
Within each road, vehicles move at the free-flow velocity, while traffic flows 
are exchanged between adjacent roads by means of the signalized 
intersection connecting them. 
In this section, we discuss a concise dynamical model for traffic networks in 
regimes of free flow, that will be employed for the analysis.

\subsection{Model of Road and Traffic Flow}
Let $\mc N = (\mc R, \mc I)$ denote a traffic network
with roads $\mc R = \{r_1,\dots,r_{\subscr{n}{r}}\}$ and intersections 
$\mc I = \{\mc I_1, \dots, \mc I_{n_{\mc I}} \}$. 
Each element in the set $\mc R$ models a one way road, 
whereas intersections regulate conflicting flows of traffic
among adjacent roads (see Section \ref{subsec: intersections}).
We assume that exogenous inflows enter the network at (source) 
roads $\mc S \subseteq \mc R$ and, similarly, vehicles 
exit  the  network at (destination) roads $\mc D \subseteq \mc R$, 
with $\mc S \cap \mc D = \varnothing $.
%
The following standard connectivity assumption ensures that vehicles are 
allowed to leave the network.
\begin{assumption}
For every road $r_i \in \mc R$ there exists at least one path in $\mc N$ from 
$r_i$ to a road $r_j \in \mc D$.
\end{assumption}
\noindent
We denote by $\ell_i \in \real$ the length of road $r_i$, and
we model each road $r_i$ by discretizing it into  
$\sigma_{i} = \lceil \ell_i / h \rceil$ cells
(see Fig.~\ref{fig: road discretized}), 
where the parameter $h \in \realpos$ is a constant discretization step.
\begin{figure}[tb]
  \centering
    \includegraphics[width=.8\columnwidth]{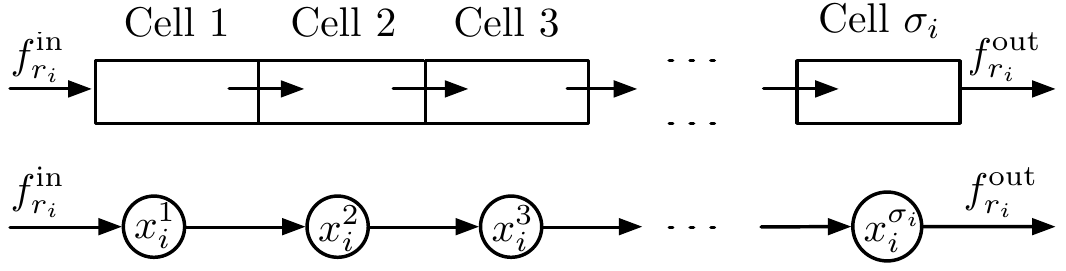}
\caption[]{Road discretization and corresponding network model. 
The portion of road comprised between spatial coordinates 
$s_{k-1}$   and $s_{k}$, $ k \in \until{\sigma_i} $ is referred to 
as cell $k$, and variable $x_i^k$ denotes the cell density.}
  \label{fig: road discretized}
\end{figure}
We denote by $x_i^k \in \real$ the traffic density\footnote{The terms density 
and occupancy will be used interchangeably in the remainder of this paper.} 
associated with the $k$-th cell of  road $r_i$, $k \in \until{\sigma_i}$.
We assume that inflows of vehicles $\supscr{f}{in}_{r_i}$ 
enter the road in correspondence of its upstream cell (i.e. $k=1$);
accordingly,  outflows $\supscr{f}{out}_{r_i}$ leave the road in
correspondence of its downstream cell (i.e. $k=\sigma_i$).
We model the relation between traffic flows and cell density by 
assuming that flows of vehicles move in regimes of free flow with constants 
velocity $\gamma_i$, which represents the average speed along link $r_i$. 
Then, the  dynamics of the road state
$x_i = [x_i^1 \, \cdots \, x_i^{\sigma_i}]^\transpose$ are described by:
\begin{align}
\label{eq: road dynamics}
  \begin{bmatrix}
    \dot x_i^1 \\ \dot  x_i^2 \\ \vdots \\ \dot  x_i^{\sigma_i}
  \end{bmatrix}
  =
  \underbrace{
  \frac{\gamma_i}{h}
  \begin{bmatrix}
    -1 &  \\
    1 & \negat 1 &  \\
    & \ddots & \ddots  & \\
    & & 1 & 0\\
  \end{bmatrix}}_{D_i}
  \begin{bmatrix}
    x_i^1 \\ x_i^2 \\ \vdots \\ x_i^{\sigma_i}
  \end{bmatrix}
  +
  \begin{bmatrix}
    \supscr{f}{in}_{r_i}   \\ 0 \\ \vdots \\ -\supscr{f}{out}_{r_i}
  \end{bmatrix}.
\end{align}
Differently from conventional dynamical models for urban links
(e.g. \cite{daganzo1994cell}), the space discretization technique 
\eqref{eq: road dynamics} allows us to capture the fact that the density of 
vehicles may not be uniform along the road. 
Notably, this feature plays a key role in modeling the road outflows 
in correspondence of signalized intersections 
(see Section~\ref{subsec: intersections}).

\begin{figure}[tb]
  \centering
    \includegraphics[width=.6\columnwidth]{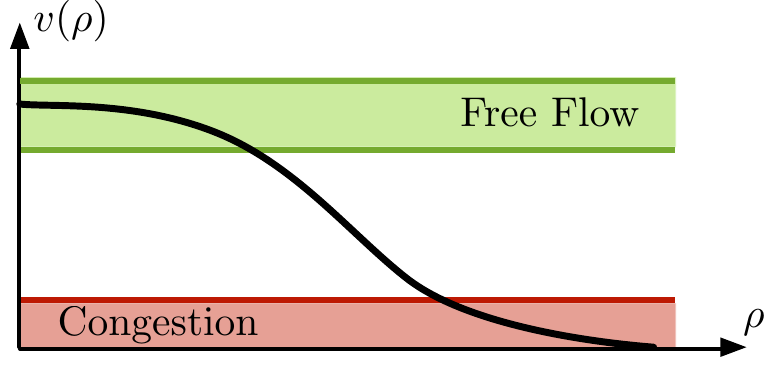}
  \caption[]{Fundamental diagram describing the static speed-density  
  relation. 
The almost-flat behavior in regimes of free-flow and heavy congestion allows 
us to approximate $\gamma ( \rho_i ) \approx \gamma_i$, as the speed of the 
of flow is  not sensibly affected  by changes in density.}
  \label{fig: speed-density relation}
\end{figure}

\begin{remark}{\bf \emph{(Equivalence Between
\eqref{eq: road dynamics} and Hydrodynamic Models)}}
\label{rem: hydrodynamic model}
The dynamical model \eqref{eq: road dynamics}
derives from the mass-conservation continuity equation 
\cite{treiber2013traffic} in certain traffic regimes, as we explain next.
Let the function $\rho_{i} = \rho_{i}(s,t) \geq 0$ denote the 
(continuous) density  of vehicles within road $r_i$ at the spatial 
coordinate $s \in [0,\ell_i]$ and time $t \in \realpos$.
Let $f_{i} = f_{i}(s,t) \geq 0$ denote the (continuous) 
flow of vehicles along the road, and let traffic densities and flows
follow the hydrodynamic relation 
 \begin{align*}
 \frac{ \partial \rho_{i}}{ \partial t} + 
 \frac{ \partial f_{i}}{ \partial s} = 0.
 \end{align*}
We first complement the above continuity equation with
the Lighthill-Whitham-Richards static relation,  $f_{i} = f_i(\rho_{i}) $, 
in which traffic flows instantaneously change with the density.
Then, we include the speed-density fundamental relation  
$f_{i} = \rho_{i}  v(\rho_{i}) $, where  
$\map{v}{\realpos}{\realpos}$ represents the speed of the traffic 
flow (see Fig.~\ref{fig: speed-density relation}), to obtain
\begin{align*}
 \frac{ \partial \rho_{i}}{ \partial t} +  \left( v(\rho_{i}) 
 + \rho_{i}   \frac{d ~ v(\rho_{i})}{d \rho_{i}} \right) 
 \frac{\partial \rho_{i}}{\partial s} = 0.
\end{align*}
Solutions to the above relation are kinematic waves 
\cite{lighthill1955kinematic} moving at speed 
$\gamma (\rho_{i}) = v(\rho_{i})  
+ \rho_{i} \frac{d v(\rho_{i})}{d \rho_{i}} $.
We consider regimes of free flow where the speed of the kinematic wave can 
be approximated as $\gamma ( \rho_i ) \approx \gamma_i$.
As illustrated in Fig.~\ref{fig: speed-density relation}, this 
approximation is accurate in regimes of free flow or  
congestion, characterized by 
$\frac{d v(\rho_{i})}{d \rho_{i}} \approx 0$.
Therefore, we let $\gamma_i$ denote the average speed of the flow 
along  the link, and consider the approximated continuity equation
\begin{align*}
 \frac{ \partial \rho_{i}}{ \partial t} + 
  \gamma_i ~ \frac{ \partial \rho_{i}}{ \partial s} = 0.
\end{align*}
We then discretize in space the above linear continuity equation, by  defining  
the discrete spatial coordinate
\begin{align*}
s_k = k h, & & k \in \{ 0 , \dots ,\sigma_i \},
\end{align*}
and by replacing the partial derivative with respect to $s$ with 
the difference quotient
\begin{align*}
 \frac{ \partial \rho_{i}(s_k,t)}{ \partial t} = 
-  \gamma_i \; \frac{\rho_{i}(s_{k},t) - \rho_{i}(s_{k-1},t)}{h} 
 .
\end{align*}
This discretization leads to the dynamical model  
\eqref{eq: road dynamics} after introducing the boundaries 
conditions, represented by inflows $\supscr{f}{in}_{r_i}$ and
outflows $\supscr{f}{out}_{r_i}$, and by replacing the spatially 
discretized density $\rho_{i}(s_k,t)$  with the functions of time  $x_i^k$.
  \oprocend
\end{remark}

\subsection{Model of Intersection and Interconnection Flow}
\label{subsec: intersections}
Signalized intersections alternate the right-of-way of vehicles to
coordinate and secure conflicting flows between adjacent roads. 
Every signalized intersection $\mc I_j \in \mc I$, $j \in \until{n_{\mc I}}$, 
is modeled as a set $\mc I_j \subseteq \mc R \times \mc R$, 
consisting of all allowed movements between the intersecting roads.
For road $r_i \in \mc R$, let $\subscr{\mc I}{in}^{r_i}$ denote the (unique) 
intersection at the road upstream; similarly, let  $\subscr{\mc I}{out}^{r_i}$ 
denote the (unique) intersection at the road downstream.
We model the effect of signalized intersections through a set of green split 
functions 
$\map{s}{\mc R \times \mc R \times \realpos}{\{0,1\}}$
that assume boolean values $1$ (green phase) or $0$ (red phase), and
let the road inflows be
\begin{align}\label{eq: interconnection flows}
  \begin{split}
    \supscr{f}{in}_{r_i} &= \sum_{(r_i,r_k) \in \subscr{\mc I}{in}^{r_i}} 
    s(r_i,r_k, t) \,
    f(r_i , r_k ) + u_{r_i} ,\\
    \supscr{f}{out}_{r_i} &= \sum_{(r_k,r_i) \in \subscr{\mc I}{out}^{r_i}}
     s(r_k,r_i, t) \,
    f(r_k, r_i ) + w_{r_i},
  \end{split}
\end{align}
where $\map{f}{\mc R \times \mc R}{\realpos}$ denotes the intersection
transmission rate. 
It is worth noting that the notation $f (r_i , r_k )$ represents the transmission 
rate from road $r_k$ to  $r_i$ and, similarly, $s(r_i,r_k, t)$ denotes the green 
split function that controls traffic flows from $r_k$ and directed to  $r_i$.
In \eqref{eq: interconnection flows}, we incorporated exogenous inflows and 
outflows to roads (flows that are not originated or merge to modeled 
intersections or roads) into  
$ \supscr{f}{in}_{r_i}$ and $ \supscr{f}{out}_{r_i}$ respectively, 
by  defining the input term $\map{u_{r_i}}{\realpos}{\realpos}$ and  
output term $\map{w_{r_i}}{\realpos}{\realpos}$. 
We note that $u_{r_i} \neq 0$ if and only if $r_i \in \mc S$, and
$w_{r_i} \neq 0$ if and only if $r_i \in \mc D$.

\begin{figure}[t]
  \centering
    \includegraphics[width=\columnwidth]{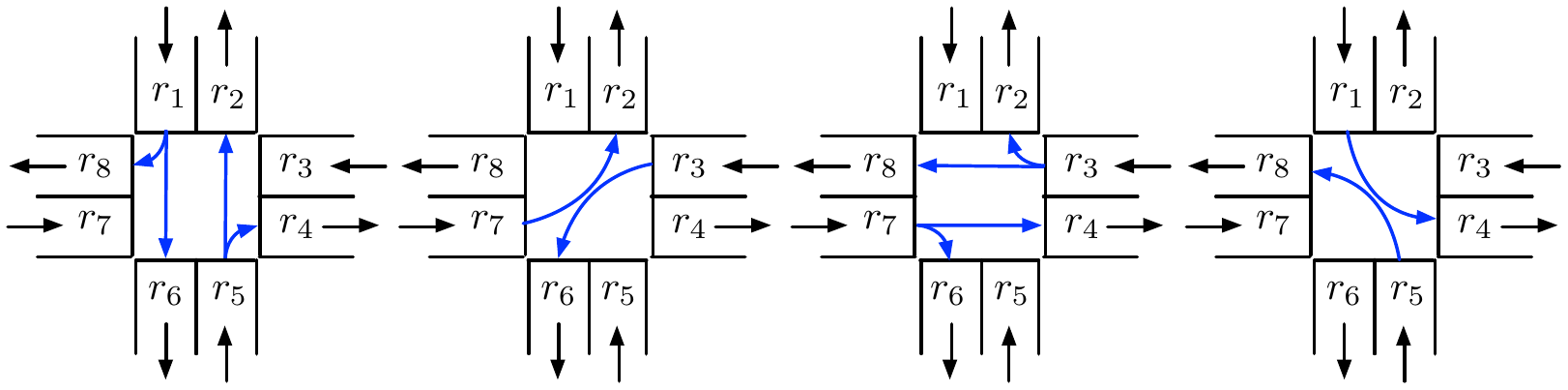}
\caption[]{Typical set of phases at a four-ways intersection.}
\label{fig: phases}
\end{figure}
\begin{figure}[t]
\centering
\includegraphics[width=.8\columnwidth]{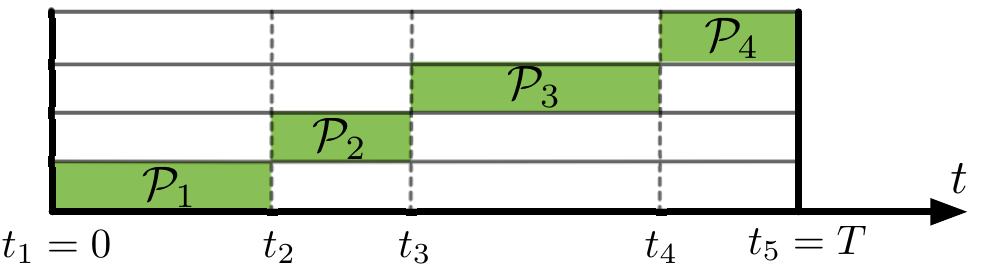}
\caption[]{Green splits for the four ways intersection in 
Fig.~\ref{fig: phases}. 
}
\label{fig: green splits}
\end{figure}

\begin{example}{\bf \emph{(Example of Intersections and Scheduling 
Functions)}}
\label{example: scheduling functions}
Consider the four-ways intersection illustrated in Fig.~\ref{fig: phases}.
The intersection is modeled through the set of allowed movements:
\begin{align*}
\mc  I_1 = 
	\{(r_1&,r_6),(r_1,r_8),(r_5,r_2),(r_5,r_4),(r_7,r_2),(r_3,r_6), \\
	&(r_3,r_8),(r_3,r_2),(r_7,r_4),(r_7,r_6),(r_5,r_8),(r_1,r_4) \}.
\end{align*}
Intersections are often partitioned into a group of phases, where each phase 
represents a set of movements that can occur simultaneously across the 
intersection. For intersection $\mc I_1$, a typical set of phases is
\begin{align*}
\mc P_1 &= \{(r_1,r_6),(r_1,r_8),(r_5,r_2),(r_5,r_4)\} , \\
\mc P_2 &= \{(r_7,r_2),(r_3,r_6)\}, \\
\mc P_3 &= \{(r_3,r_8),(r_3,r_2),(r_7,r_4),(r_7,r_6)\}, \\
\mc P_4 &= \{(r_5,r_8),(r_1,r_4)\}. &
\end{align*}
The green split function is then defined by alternating the available phases 
within a cycle time $T \in \realpos$, that is, for all 
$j \in \until 4$,
\begin{align*}
s(r_i,r_k,t) &=
  \begin{cases}
    1 & \text{if }  (r_i,r_k) \in \mc P_j  \text{ and } t \in [t_j,t_{j+1}),\\
    0 & \text{otherwise},
  \end{cases}
\end{align*}
where $t_j \in \realpos$ are the switching times.
See Fig.~\ref{fig: green splits} for a graphical illustration.
\oprocend
\end{example}


Next, we make critical use of the underlying road-discretization technique 
\eqref{eq: road dynamics}, and model transmission rates as functions 
proportional to the occupancy of the cell at downstream of each road, that is,
\begin{align}\label{eq: flows}
  f(r_i , r_k ) = c(r_i, r_k) x_{k}^{\sigma_k} ,
\end{align}
where $\map{c}{\mc R \times \mc R}{\real}$ is a parameter that
models the speed of the outflow, and includes the 
average routing ratio of vehicles  entering road $r_i$ from $r_k$. 
Although \eqref{eq: flows} constitutes an approximation of the intersection 
transmission rate, which often contains saturation terms 
(e.g. see \cite{allsop1971delay} for a discussion), the combination of
\eqref{eq: flows} with the proposed space-discretization technique 
\eqref{eq: road dynamics} performs well in practice (see 
Section~\ref{sec: simulations} for numerical validation of \eqref{eq: flows}), 
at least in the considered regimes.
\begin{remark}{\bf \emph{(Turning Rates and Conservation of Flows at 
Intersections)}}
\label{rem: turning rates}
The existence of independent lanes for different turning preferences in
proximity of intersections often leads to the decomposition
$c(r_i, r_k) = \varphi(r_i, r_k) \phi(r_i, r_k)$, where the function
$\map{\varphi}{\mc R \times \mc R}{[0,1]}$ represents the constant
average routing ratio of vehicles entering road $r_i$ from $r_k$.
The conservation of flows at intersections can be guaranteed by letting 
$\sum_i \varphi(r_i,r_k) = 1$.
 \oprocend
\end{remark}
We emphasize that the proposed model represents a simplified 
framework with respect to traditional, more-established,  models. 
In fact, the assumptions we make are realistic for systems that are operating 
close to their free-flow regime, and do not take into account collateral 
phenomena, such as back-propagation in regimes of congestion, which would
require ad-hoc adjustments in the model.


\subsection{Switching and Time-Invariant Traffic Network Dynamics}
\label{sec: state space averaging}
Individual road dynamics can be combined into a joint network model 
that captures the interactions among all modeled routes and intersections.
To this aim, we adopt an approach similar to \cite{de2010multi}, and assume 
that exogenous outflows are proportional to the number of vehicles in the road, 
that is, $w_{r_i} = \bar w_{r_i} x_i^{\sigma_i}$, $\bar w_{r_i} \in [0,1]$.
By combining Equations \eqref{eq: road dynamics},
\eqref{eq: interconnection flows} and \eqref{eq: flows}, we obtain
\begin{align}\label{eq: switching model}
  \begin{split}
    \underbrace{
      \begin{bmatrix}
        \dot x_1 \\ \dot x_2 \\ \vdots \\ \dot x_{\subscr{n}{r}}
      \end{bmatrix}}_{\dot x}
    =
    \underbrace{
      \begin{bmatrix}
        A_{11} & A_{12} & \cdots & A_{1\subscr{n}{r}} \\
        A_{21} & A_{22} & \ddots & A_{2\subscr{n}{r}} \\
        \vdots & \ddots & \ddots & \vdots \\
        A_{\subscr{n}{r} 1} & A_{\subscr{n}{r} 2} & \cdots & A_{
          \subscr{n}{r} \subscr{n}{r}} \\
      \end{bmatrix}}_{A}
    \underbrace{
      \begin{bmatrix}
        x_1 \\ x_2 \\ \vdots \\ x_{\subscr{n}{r}}
      \end{bmatrix}}_{x}
          +
\quad \quad \quad \quad \quad
          &
    \\ 
    +
    \underbrace{
      \begin{bmatrix}
        I_{n_1} & 0 & \cdots & 0\\
        0 & I_{n_2} & \ddots & 0 \\
        \vdots & \ddots & \ddots & \vdots \\
        0 & 0 & \cdots &  I_{\subscr{n}{r}} \\
      \end{bmatrix}}_{B}
    \underbrace{
      \begin{bmatrix}
        u_1 \\ u_2 \\ \vdots \\ u_{\subscr{n}{r}}
      \end{bmatrix}}_u,
      &
  \end{split}
\end{align}
where $A \in \real^{n \times n}$, $n = \sum_{i=1}^{\subscr{n}{r}} \sigma_i$ is 
the overall number of modeled network cells, $u$ derives from 
\eqref{eq: interconnection flows}, and
\begin{align*}
  A_{ik} \!\! &= \!\!
  \begin{cases}
    s(r_i, r_k, t ) c(r_i, r_k) e_1 e_{\sigma_k}^\transpose, 
    				& \text{if }  i\neq k,\\
    D_i \! - \! \left( 
    		\sum_{\ell} s(r_\ell , r_i, t ) c(r_\ell, r_i) + \bar w_{r_i} \right)
    e_{\sigma_i}  e_{\sigma_i}^\transpose , & \text{if } i=k,
  \end{cases}
\end{align*}
where $e_i  \in \real^{\sigma_i}$ denotes the $i$-th canonical vector of size 
$\sigma_i$.

We note that the matrix $A$ in \eqref{eq: switching model} 
is typically sparse, because not all roads are adjacent in the network
interconnection, and its sparsity pattern varies over time as determined by the 
green splits  $s(r_i, r_k, t )$.
Thus, the network model \eqref{eq: switching model} represents a linear 
switching system, where the switching signals are the green split 
functions.

\begin{figure}[tb]
  \centering \subfigure[]{
    \includegraphics[width=.46\columnwidth]{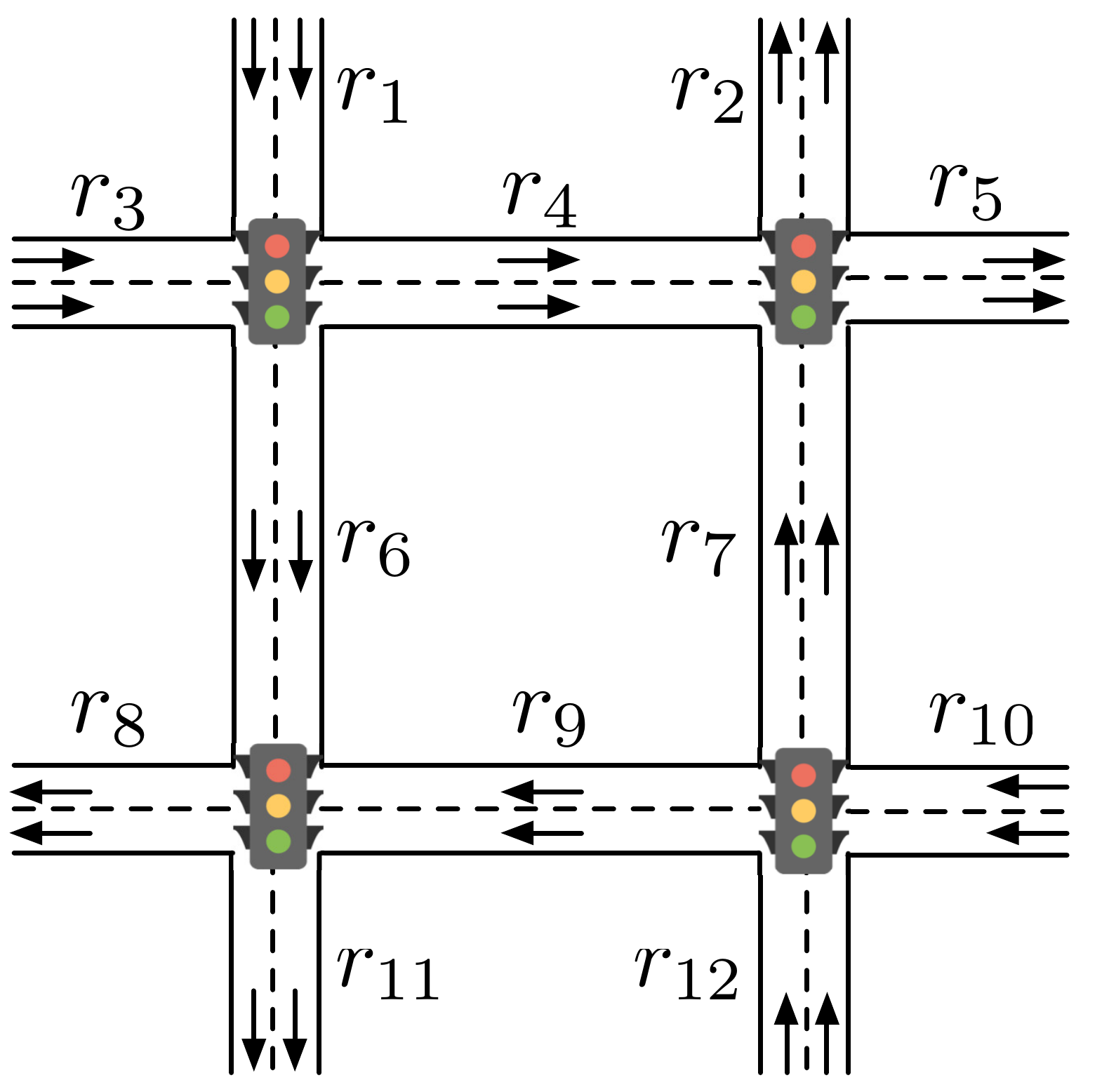} 
    \label{fig: 4 int netw a}  } 
 \subfigure[]{
    \includegraphics[width=.46\columnwidth]{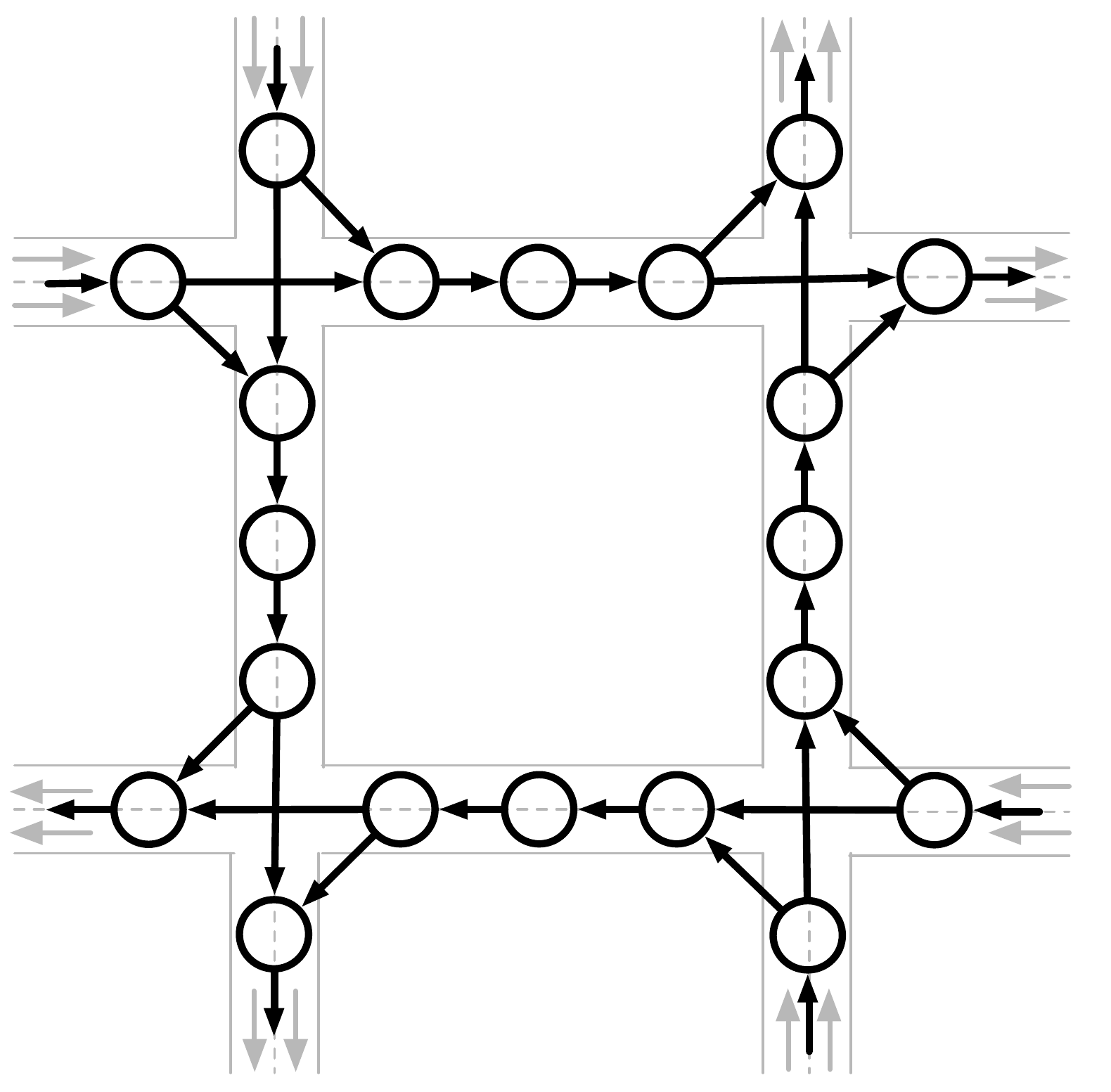}  
    \label{fig: 4 int netw b}}
  \caption[]{Network model associated with a traffic network composed of  
  $n_{\mc I}=4$ intersections and   $\subscr{n}{r}=12$ roads. 
  Each road is associated with a set of states that represent the density of the 
  cells within the roads.}
\label{fig: 4 int netw}
\end{figure}

\begin{example}{\bf \emph{(Example of Traffic Network)}}
\label{example: network dynamics}
Consider the network illustrated in Fig. \ref{fig: 4 int netw},
with $\mc R = \{ r_1, \dots , r_{12} \}$ and $\mc I = \{ \mc I_1, \dots , \mc I_4 \}$. 
The network comprises four destination roads 
$\mc D = \{r_2, r_5, r_8, r_{11} \}$ 
($\bar w_{r_i} = 1$ for all $r_i \in \mc D$, and $\bar w_{r_i} = 0$ otherwise), 
and four source roads 
($\mc S = \{r_1, r_3, r_{10}, r_{12} \}$, with $u_{r_i} \neq 0$ only if  
$r_i \in \mc S$).
Let $\ell_i/h=3$ and $\gamma_i/h = 3$ for all $i \in \until{n_r}$.
Then, the matrices in \eqref{eq: switching model} read as
\begin{align*}
A_{ii} &= 
\begin{bmatrix}
-1 \\
1 & -1 \\
   & 1  & -\left( \sum_{\ell} s(r_\ell,r_i,t) c(r_\ell,r_i) + \bar w_{r_i} \right)
\end{bmatrix}, \\
A_{ij} &= 
\begin{bmatrix}
0 & \quad \; 0 & \quad \quad \quad  \;s(r_i,r_j,t) c(r_i,r_j) \quad \quad \quad 
	\quad \\
0 & \quad \; 0 & 0\\
0 & \quad \; 0 & 0
\end{bmatrix},
\end{align*}
for all $i, j \in \until{\subscr{n}{r}}$. 
Notice that $s(r_i,r_j,t) = 0$  for all times if $(r_j,r_k) \not \in \mc I_k$
for all $k \in \until {n_{\mc I}}$.

\oprocend
\end{example}

Next, we make the typical assumption that scheduling functions are periodic, 
with period $T \in \real_{>0}$. 
That is, for all $(r_i,r_k) \in \mc I_j$, 
$j \in \until{n_{\mc I}}$, and for all times $t$:
\begin{align*}
  s(r_i,r_k,t) = s(r_i,r_k,t + T) .
\end{align*}
Let $\mc T = \{\tau_1, \dots, \tau_m\}$ denote the set of time instants
 when a scheduling function changes its value, that is,
\begin{align*}
  \mc T = \setdef{\tau \in [0,T]  }{  \exists (&r_i,r_k) \in \mc I,\\
 &  \lim_{t \rightarrow \tau^-} \!\!
  s(r_i,r_k,t) \!\neq \!\!\lim_{t \rightarrow \tau^+}  \!\!
  s(r_i,r_k,t)} .
\end{align*}
Notice that the network matrix A in \eqref{eq: switching model} 
remains constant between any  consecutive time instants  
$\tau_{i-1}$ and $\tau_i$. 
We denote each constant matrix by $A_i$, and refer to it as to the $i$-th 
\textit{network mode}. 
Further, let $d_i = \tau_i - \tau_{i-1}$, with $i \in \until m$ and 
$\tau_0 = 0$, denote the duration of the $i$-th network mode. 
We employ a state-space averaging technique \cite{middlebrook1976general}
and define a linear-time invariant approximation of the switching 
network model  \eqref{eq: switching model}:
\begin{align}
\label{eq: average system}
  \subscr{\dot x}{av} = \subscr{A}{av} \subscr{x}{av} +
  B \subscr{u}{av},
\end{align}
where $\subscr{A}{av} = \frac{1}{T}\sum_{i=1}^m d_i A_i$, and
$\subscr{u}{av} = [\subscr{u}{av,$1$} \dots \subscr{u}{av,$\subscr{n}{r}$}]$,
$\subscr{u}{av,$i$} = (1/T ) \int_{0}^{T} u_i (\tau) ~d \tau$.
We note that the averaging technique preserves the sparsity 
pattern of the network, that is, $\subscr{A}{av}(i,j) \neq 0$ if and only if 
$A_k(i,j) \neq 0$ for some  $k$.

In general, the approximation of the behavior of the switching system
\eqref{eq: switching model} with  the  average dynamics  
\eqref{eq: average system} is accurate if the operating  period $T$ is short in 
comparison to the underlying system dynamics.
Under suitable technical assumptions, the deviation of average models 
with respect to the network instantaneous state has been characterized 
in e.g. \cite{canudas2017average}.
In \cite{canudas2017average}, the authors formally provide a bound on the 
accuracy of the approximation, where the bound becomes tighter for 
decreasing values of $T$, and increasing values of road lengths.
Although a formal characterization of the performance of average model is 
often an extremely challenging task, and the formal results in 
\cite{canudas2017average} hold for a single roads in the presence of a single 
signalized intersections, the accuracy of averaging techniques has been shown 
to be satisfactory in similar applications (see e.g. \cite{grandinetti2018}).
 %
A numerical validation of the averaging technique, and its 
validity in relation to the cycle-time $T$, is discussed in 
Section~\ref{sec: simulations} (see Fig.~\ref{fig:averageVsSwitching}).


\subsection{Problem Formulation}
\label{subsec:problemFormulation}
In this paper, we consider the dynamical model 
\eqref{eq: average system} and focus on the problem of designing 
the durations of the green split functions so that a  measure of 
network  congestion is  optimized.
As illustrated through the relation 
$\subscr{A}{av} = \frac{1}{T}\sum_{i=1}^m d_i A_i$ 
(see \eqref{eq: average system}), the average model allows us to design the 
durations  of the network modes, rather than their exact  sequence.
This approach motivates the adoption of a two-stage optimization 
process. 
First, the durations of the modes is optimized by considering a joint model that 
captures the dynamics of the entire interconnection.
Second, offset control techniques (see e.g. \cite{coogan2017offset}) can be 
employed to decide the specific sequence of phases, given the durations of 
the splits and by considering local interconnection models.
In this paper, we focus on the first stage.
To formalize our optimization problem, we denote by $\subscr{y}{av}$ the 
vector of the queue lengths originated by the signalized intersections, and 
model $\subscr{y}{av}$ as the occupancy of the cells at roads downstream, 
that is,
\begin{align}
\label{eq:queueLengths_y_Cav}
\subscr{y}{av} = \subscr{C}{av} \subscr{x}{av}, \;
& &  \;
\subscr{C}{av} = 
\begin{bmatrix}
e_{\sigma_1}^\tsp & \dots & 0 \\
\vdots & \ddots \\
0 & \dots & e_{\sigma_{\subscr{n}{r}}}^\tsp
\end{bmatrix} .
\end{align}

We assume the network is initially at a certain initial state $x_0$, and 
focus on the problem of optimally designing the mode durations 
$\{ d_1, \dots d_m \}$ that  minimize the $\mc H_2$-norm
of the vector of queue lengths $\subscr{y}{av}$.
Namely, we consider the following dynamical optimization problem:
\begin{subequations}\label{opt: informal}
\begin{align}
 \min_{d_1,\dots,d_m} \;\;\;\;\;\;
  		&
\int_0^\infty \| \subscr{y}{av} \|_2^2 ~ dt,	
  		\nonumber \\[.3em]
  \text{subject to} \;\;\;\;
  &\subscr{\dot x}{av} =\subscr{A}{av} \subscr{x}{av}, 
  			\label{eq: constr a}\\
	&\subscr{y}{av} = \subscr{C}{av}\subscr{x}{av}, 	
			\label{eq: constr b} \\
&\subscr{x}{av}(0) = x_0, \label{eq: constr d} \\
&\subscr{A}{av} = \frac{1}{T} \left( d_1A_1 + \dots + d_m A_m \right),
		\label{eq: constr c} \\
& T = d_1 + \dots + d_m,\label{eq: constr e} \\
&d_i \ge 0 \quad  i\in \until{m} \label{eq: constr f}.
\end{align}
\end{subequations}
Loosely speaking, the optimization problem \eqref{opt: informal} seeks for 
an optimal set of split durations that minimize the $\mc L_2$-norm of the 
impulse-response of the system to the initial conditions $x_0$.
Thus, similarly to \cite{gomes2006optimal}, our framework considers the 
``cool down'' period. 
In this scenario, exogenous inflows and outflows are not 
known a priori, and the goal is to evacuate the network as fast as 
possible, where the final condition is an empty system.
Under this assumption, $\subscr{A}{av}$ and the solution to  
\eqref{opt: informal}, shall be re-computed when the current traffic conditions 
$x_0$ change significantly.
Finally, we note that constraint \eqref{eq: constr e} guarantees
feasibility of the resulting solutions. 
In fact, the time-normalization in \eqref{eq: constr e} implies that for any 
solution resulting from \eqref{opt: informal}, there exists (at least) one 
sequence of movements at each intersection, with overall cycle time $T$, and 
that abide the selected set of durations.

\section{Design of Optimal Network Mode Durations} 
\label{sec: optimal design}
In this section, we propose a numerical method to determine solutions to the 
optimization problem \eqref{opt: informal}.
The approach we discuss is centralized, namely, it requires 
knowledge of the state $x_0$ for all network cells. 
An extension of the framework to fit a distributed implementations is then 
discussed in Section~\ref{sec:distributed}.
At its core, the technique relies on rewriting the cost 
function of \eqref{opt: informal} in relation to the controllability Gramian of an 
appropriately-defined linear system, as we explain next. 

\begin{lemma}{\bf \textit{(Controllability Gramian Cost Function)}}
\label{lemma: gramian opt}
Let 
\begin{align*}
\mc W(\subscr{A}{av}, x_0)  = 
\int_0^\infty   
e^{\subscr{A}{av} t} 		x_0 	  	x_0^\tsp 		
e^{\subscr{A}{av}^\tsp   t 		} 		~~ dt .
\end{align*}
The following minimization problem is equivalent to 
\eqref{opt: informal}:
\begin{align}
\label{opt: gramian cost function}
\underset{d_1,\dots ,d_m}{\text{ min }}     \;\;\;& \quad 
		\trace 		\left(  \subscr{C}{av}  	~ \mc W(\subscr{A}{av}, x_0)  		
		~ \subscr{C}{av}^\tsp 		\right), 	   \nonumber \\
\text{ subject to} & \quad
\subscr{A}{av} = \frac{1}{T} \left( d_1A_1 + \dots + d_m A_m \right), 
\nonumber \\
& \quad T = d_1 + \dots + d_m, \nonumber \\
& \quad d_i \ge 0, \quad i\in \until{m}.
\end{align}

\end{lemma}
\smallskip

\begin{proof}
By incorporating \eqref{eq: constr a}, \eqref{eq: constr b}, and 
\eqref{eq: constr d} into the cost function of optimization problem 
\eqref{opt: informal}, we can rewrite:
\begin{align*}
 \int_0^\infty \|  \subscr{y}{av} \|_2^2 ~ dt 
 & =  		
\int_0^\infty   	x_0^\tsp 	e^{\subscr{A}{av}^\tsp t} 	
\subscr{C}{av}^\tsp  		\subscr{C}{av} 	e^{\subscr{A}{av} t} 	
x_0 	~ dt
& \\
&=
\trace 	\left(  	\int_0^\infty  	x_0^\tsp 	e^{\subscr{A}{av}^\tsp 
t} 
\subscr{C}{av}^\tsp  		\subscr{C}{av}  		e^{\subscr{A}{av} t}  	
x_0 		~ dt   \right) & \\
&= 
 \int_0^\infty  	\trace 		\left( 	\subscr{C}{av} 		
 e^{\subscr{A}{av} t}   	x_0  		x_0^\tsp 		e^{\subscr{A}{av}^
 \tsp t } 	
 \subscr{C}{av}^\tsp 		 \right)  dt & \\
 &=  
 \trace  		\left(   		\subscr{C}{av} 		\int_0^\infty   
e^{\subscr{A}{av} t} 		x_0 		x_0^\tsp 
e^{\subscr{A}{av}^\tsp t}  		dt  ~ 		\subscr{C}{av}^\tsp  	
\right) ,
&
\end{align*}
from which the claimed statement follows.
\end{proof}
Next, we note that the cost function of \eqref{opt: informal} is finite only if the 
choice of parameters $\{ d_1, \dots , d_m \}$ leads to a dynamical 
matrix $\subscr{A}{av}$ that is  Hurwitz-stable.
Requiring Hurwitz-stability of $\subscr{A}{av}$ corresponds to imposing
$\alpha (\subscr{A}{av}) <0$,  where $\alpha (\subscr{A}{av}) := 
\sup \{ \Re (s) : s \in \complex, \det ( s I - \subscr{A}{av} ) =0 \}$
denotes the spectral abscissa of $\subscr{A}{av}$.
The following result proves network stability under optimal sets of split 
durations.
\begin{theorem}{\bf \textit{(Stability of Optimal Solutions)}}
\label{lemma: feasibility}
Consider a traffic network $\mc N = (\mc R, \mc I)$, 
and assume that there exists a path from every node in $\mc N$ to at least 
one node in $\mc D$.
Let $s(r_i,r_k,\bar t) \neq 0$ for all $(r_i,r_k) \in \mc I$, and for some 
$\bar t \in [0,T]$.
Then,
\begin{align*}
\alpha ( \subscr{A}{av} ) < 0.
\end{align*}
\end{theorem}
\begin{proof}
From the structure of \eqref{eq: switching model} and from the assumption
$s(r_i,r_k,\bar t) \neq 0$ follows that  $\subscr{A}{av}(i,i)<0$  for all  
$i \in \until n$,  while  $\subscr{A}{av}(i,j) \geq 0$ for all $j \in \until n$, 
$j \neq i$.
Moreover, all columns of $\subscr{A}{av}$ have nonpositive sum. 
In particular, the columns corresponding to destination cells have strictly 
negative sum, that is, 
$\sum_{i=1}^n \subscr{A}{av}(i,j) \leq 0$ for all $j \in \until n$, 
and $\sum_{i=1}^n \subscr{A}{av}(i,j) < 0$ for all $j$ such that $r_j \in \mc D$.
%
%
To show $\alpha ( \subscr{A}{av} ) < 0$, we use the fact that destination cells 
in $\mc D$ have no departing edges, and re-order the states so that
\begin{align*}
\subscr{A}{av} = 
\begin{bmatrix}
A_{11} & 0 \\
A_{21} & A_{22} 
\end{bmatrix},
\end{align*}
where $A_{22} \in \real^{n_d \times n_d}$, $n_d = | \mc D |$, 
is the submatrix that describes the dynamics of the destination cells, 
$A_{11} \in \real^{(n-n_d) \times (n-nd)}$, and 
$A_{21} \in \real^{n_d \times (n-nd)}$.
The fact $\alpha (A_{22})<0$ immediately follows from 
\eqref{eq: interconnection flows}.
The stability of $A_{11}$ follows from the connectivity assumption in the 
original network, and from the analysis of  grounded Laplacian matrices 
(see e.g. \cite[Theorem 1]{xia2017analysis}).
\end{proof}
%
%

%
For the solution of the optimization problem
\eqref{opt: gramian cost function}, we propose a method based on a 
generalization of the concept of spectral  abscissa, namely, the 
smoothed spectral  abscissa.
For a dynamical system of the form 
\eqref{eq: average system}-\eqref{eq:queueLengths_y_Cav}, 
the  smoothed spectral abscissa  \cite{vanbiervliet2009smoothed}
is defined  as  the  root $\tilde \alpha \in \real$ of  the implicit equation
\begin{align}
\label{eq: e-smoothed sp abs}
\trace \left( \subscr{C}{av} 
~ \mc W( \subscr{A}{av} -\tilde \alpha I, B) ~
 \subscr{C}{av}^\tsp \right)  = \epsilon^\inv,
\end{align}
where $\epsilon \in \real_{\geq 0}$.
It is worth noting that the root $\tilde \alpha$ is  unique 
\cite{vanbiervliet2009smoothed}, and for fixed $B$ and $\subscr{C}{av}$ it is 
a function of both $\epsilon$ and $\subscr{A}{av}$, namely 
$\tilde \alpha(\epsilon, \subscr{A}{av})$.

\begin{remark}{\bf \textit{(Properties of the Smoothed Spectral Abscissa)}}
For any $\epsilon > 0$, the smoothed spectral abscissa is 
an upper  bound to $\alpha (A)$, and this bound becomes  exact as 
$\epsilon  \rightarrow 0$.
To see this, we first observe that the integral
\sloppy{
$ \int_0^\infty e^{(\subscr{A}{av} -\tilde \alpha I)t} B 
B^\tsp e^{(\subscr{A}{av}-\tilde \alpha I)^\tsp t} dt $
}
exists and is finite for any  $\tilde \alpha > \alpha ( \subscr{A}{av} )$, 
as the function $e^{( \subscr{A}{av} -\tilde \alpha I) t}$ is bounded and 
convergent as $t \rightarrow + \infty$.
On the other hand, for any $\tilde \alpha < \alpha ( \subscr{A}{av} )$
 the function $e^{( \subscr{A}{av} -\tilde \alpha I) t}$ becomes unbounded for
$t \rightarrow + \infty$ and the above integral is infinite.
It follows that, the left-hand side of \eqref{eq: e-smoothed sp abs} is finite
only if  $\tilde \alpha > \alpha (A)$ or, in other words, for any finite $\epsilon$,  
$\tilde \alpha $ satisfies  $\tilde \alpha > \alpha (A)$.
\oprocend
\end{remark}

We observe that, by letting  $\tilde \alpha =0$ in   
\eqref{eq: e-smoothed sp abs}, the optimization 
problem \eqref{opt: gramian cost function} can equivalently be reformulated 
in terms of the smoothed spectral abscissa as follows:
\begin{align}
\label{opt: epsilon cost function}
\underset{d_1,\dots ,d_m, \epsilon }{\text{ min }} 
		\;\;\;& \quad \epsilon^\inv, \nonumber \\
\text{ subject to} & \quad
\subscr{A}{av} = \frac{1}{T} \left( d_1A_1 + \dots + d_m A_m \right), 
\nonumber \\
& \quad T = d_1 + \dots + d_m, \nonumber \\
& \quad d_i \ge 0, \quad i\in \until{m}, \nonumber \\
& \quad \tilde \alpha( \epsilon, \subscr{A}{av}) = 0 ,
\end{align}
where the parameter $\epsilon$ is now an optimization  variable.
In what follows, we denote with  $\{ d_1^*, \dots, d_m^*, \epsilon^* \}$
the value of the optimization parameters at optimality of 
\eqref{opt: epsilon cost function}.
Problem \eqref{opt: epsilon cost function} is a nonlinear optimization 
problem \cite{vanbiervliet2009smoothed}, because the optimization variables   
$\{ d_1, \dots, d_m \}$ and  $\epsilon$ are related by means of the nonlinear  
equation \eqref{eq: e-smoothed sp abs}.

For the solution of \eqref{opt: epsilon cost function}, we propose an  iterative 
two-stages numerical optimization process. 
In the first stage, we fix the value of $\epsilon$ and seek for a 
choice of $\{ d_1, \dots d_m \}$  that leads to a smoothed spectral abscissa 
that is identically zero.
In other words, we let $\epsilon = \bar \epsilon$, and solve the 
following minimization problem:
\begin{align}
\label{opt: sp abs cost fcn, fixed epsilon}
\underset{d_1,\dots ,d_m }{\text{min }} 
		\;\;\;& \quad | \tilde \alpha( \bar \epsilon, \subscr{A}{av}) |
		 \nonumber \\
\text{ subject to} & \quad
\subscr{A}{av} = \frac{1}{T} \left( d_1A_1 + \dots + d_m A_m \right), 
\nonumber \\
& \quad T = d_1 + \dots + d_m, \nonumber \\
& \quad d_i \ge 0, \quad i\in \until{m}.
\end{align}
It is worth noting that every $\subscr{\bar A}{av}$ that is solution to
\eqref{opt: sp abs cost fcn, fixed epsilon} and that satisfies
\sloppy{$\tilde \alpha( \bar \epsilon, \subscr{\bar A}{av} )  = 0$},
is a point in the feasible set of \eqref{opt: epsilon cost function},
which corresponds to a cost of congestion 
$\int_0^\infty \| \subscr{y}{av} \|_2^2 ~ dt  = 1/{\bar \epsilon}$.

In the second stage, we perform a line-search over the parameter 
$\epsilon$, where the value of $\epsilon$ is increased at every 
iteration until the minimizer $\epsilon^*$ is achieved.
We note that the optimizer of \eqref{opt: sp abs cost fcn, fixed epsilon} with 
$\epsilon = \epsilon^*$ is indeed $\{ d_1^*, \dots , d_m^* \}$, an optimal 
solution to \eqref{opt: epsilon cost function}.
Thus, heuristically, the progressive $\epsilon$-update step in the two stages 
optimization problem allows us to determine (local) solutions to 
\eqref{opt: epsilon cost function}. 

The remainder of this section is devoted to the description of the two-stage 
optimization process.
The benefit of solving \eqref{opt: sp abs cost fcn, fixed epsilon} as opposed 
to \eqref{opt: epsilon cost function} is that we can derive an expression for the 
gradient of $\tilde \alpha_{\bar \epsilon}$ with respect to 
the mode durations $\{ d_1, \dots , d_m \}$, as we illustrate next. 
In the remainder of this section, with a slight abuse of notation, we use the 
compact form 
$\tilde \alpha(\bar \epsilon, \subscr{A}{av}) = \tilde \alpha_{\bar \epsilon}$ and,
for a matrix $M = [m_{ij}] \in \real^{m \times n}$, we denote
its vectorization by $\supscr{M}{v} = [m_{11} \dots m_{m1}, m_{12} \dots  
m_{mn}]^\tsp$.

\begin{lemma}{\bf (\textit{Descent Direction})}
\label{lemma: gradient}
Let $\tilde \alpha_{\bar \epsilon} $ denote  the  
unique root  of  \eqref{eq: e-smoothed sp abs} with 
$\bar \epsilon \in \real_{>0}$.
Let $d = [ d_1, \dots , d_m ]^\tsp$, and let
$K = [\supscr{A}{v}_1 ~ \supscr{A}{v}_2 ~ \dots ~ \supscr{A}{v}_m]$.
Then,
 \begin{align*}
 \frac{\partial \tilde \alpha_{\bar \epsilon} }{\partial d} 
  =  K^\tsp
\supscr{  \left(		\frac{QP }{		\trace  \left( 	Q	P \right)		} \right) }{v}
 \end{align*}
 where $P \in \real^{n \times n}$ and $Q \in \real^{n \times n}$ 
are the unique solution to the two Lyapunov  equations
\begin{align}
\label{eq: lyapunov P and Q}
(	 \subscr{A}{av}	- \tilde  \alpha_{\bar \epsilon} I	) ~	P 		&+
		P 	(	\subscr{A}{av} - \tilde \alpha_{\bar \epsilon} I		)^\tsp 	+
			 x_0		x_0^\tsp \; 	= 	0  , 	 &\nonumber \\
(	 \subscr{A}{av}	-	\tilde \alpha_{\bar \epsilon} I	)^\tsp  	Q	&+		 
	Q (  \subscr{A}{av}	-	\tilde  \alpha_{\bar \epsilon} I	) 	~+
		 	\subscr{C}{av} \subscr{C}{av}^\tsp 	= 	0 ,
\end{align}
and  $I \in \real^{n \times n}$ denotes the identity matrix.
\end{lemma}
\smallskip
\begin{proof}
The expression for the partial derivative of  the smoothed  spectral 
abscissa with respect to $d$  can be obtained from the composite function
\begin{align*}
 \frac{\partial \tilde \alpha_{\bar \epsilon} }{ \partial d} 
= 
\frac{\partial \subscr{A}{av}}{\partial d} ~
\frac{\partial \tilde \alpha_{\bar \epsilon}}{\partial \subscr{A}{av}} ,
\end{align*}
where   $\frac{\partial \subscr{A}{av}}{\partial d}$ follows immediately from 
 \eqref{eq: constr c}, and the expression for the derivative of 
$\tilde \alpha_{\bar \epsilon}$  with respect to $\subscr{A}{av}$  is given in 
 \cite[Theorem 3.2]{vanbiervliet2009smoothed}.
\end{proof}
We note that equations \eqref{eq: lyapunov P and Q} always admit unique 
solution. 
To see this, we use the fact that $\tilde \alpha$ is an upper bound to 
$\alpha (\subscr{A}{av})$, and observe that
$(\subscr{A}{av} - \tilde \alpha_{\bar \epsilon} I)$ is 
Hurwitz-stable for every $\subscr{A}{av}$. 

A gradient descent method based on Lemma \ref{lemma: gradient}
is illustrated in  Algorithm \ref{alg: min sm sp abs}.
Each iteration of the algorithm comprises  the following steps. 
First, (lines $4-7$) a (possibly non feasible) descent direction $\nabla$ is 
derived as illustrated in Lemma  \ref{lemma: gradient}. 
Second, (line $8-9$) a gradient-projection technique 
\cite{luenberger1984linear} is used to enforce 
constraints \eqref{eq: constr c}-\eqref{eq: constr f}.
The update-step follows (line $10$).
\begin{algorithm}[tb]
\KwIn{
Matrix $\subscr{C}{av}$, 
vector $x_0$, 
scalars  $\xi$, $\mu$}
\KwOut{
$\{ d_1^*, \dots d_m^*, \epsilon^* \}$ solution to \eqref{opt: informal} }
Initialize: 
$d^{(0)}$, 
$\bar \epsilon=0$, 
$k=1$ \\
\While{\text{$\tilde \alpha_{\bar \epsilon}^{(k)}=0$}}{
\Repeat{$ \mc P^{(k)} \nabla = 0$}{ 
Compute $\tilde \alpha_{\bar \epsilon}^{(k)}$ by solving 
\eqref{eq: e-smoothed sp abs}\;
Solve for $P$ and $Q$: 
		$(	 \subscr{A}{av}^{(k)}	- \alpha_{\bar \epsilon}^{(k)} I	) 	P 		+
		P 	(	\subscr{A}{av}^{(k)} - \alpha_{\bar \epsilon}^{(k)}  I		)^\tsp 	
		+	 x_0		x_0^\tsp 	= 	0  $; 
		$ ( \subscr{A}{av}^{(k)}	-	\alpha_{\bar \epsilon}^{(k)} I	)^\tsp  	Q	+		 
	Q (  \subscr{A}{av}^{(k)}	-	\alpha_{\bar \epsilon}^{(k)} I	) 	
		+ 	\subscr{C}{av} \subscr{C}{av}^\tsp 	= 	0 $\;
$  \frac{\partial \alpha_{\bar \epsilon}^{(k)} }{\partial d} \gets
		\frac{Q P}{\trace(QP)}$\;
$\nabla \gets  \tilde \alpha_{\bar \epsilon}
		 \frac{\partial \alpha_{\bar \epsilon}^{(k)} }{\partial d} $\;
Compute projection matrix $\mc P^{(k)}$\;
$d^{(k)}\gets  d^{(k)} - \mu ~ \mc P^{(k)} \nabla $\;
$\subscr{A}{av}^{(k)} \gets  
		\frac{1}{T} \left( d_1A_1 + \dots + d_m A_m \right)$\;
		$k \gets k+1$\;
}
$\bar \epsilon \gets \bar \epsilon + \xi$\;
}
\Return{$d$}\;
\caption{Centralized solution to \eqref{opt: informal}.}
\label{alg: min sm sp abs}
\end{algorithm}
Algorithm \ref{alg: min sm sp abs} employs a fixed stepsize 
$\mu \in (0,1)$, and a terminating criterion (line $13$) based on the
Karush-Kuhn-Tucker conditions for projection methods
\cite{luenberger1984linear}.
The $\epsilon$-update step, which constitutes the outer 
while-loop (line $2-14$), is then 
performed  at each iteration of the gradient descent  phase, and 
the line-search is terminated when $| \tilde \alpha_\epsilon |=0$
cannot be achieved.
To prevent the algorithm from stopping at local minimas, the gradient descent 
algorithm can be repeated over multiple feasible initial conditions 
$d^{(0)}$.

\begin{remark}{\bf (\textit{Complexity of 
Algorithm~\ref{alg: min sm sp abs}})}
The computational complexity of Algorithm~\ref{alg: min sm sp abs}
in relation to the network size can be derived as follows.
First, the computation of the current value of the smoothed spectral 
abscissa  can be achieved via a root-finding algorithm 
(such as the bisection algorithm) over the implicit  function  
\eqref{eq: e-smoothed sp abs}.
The complexity of this operation is a logarithmic function of the 
desired accuracy, and it requires the computation of the trace of the 
Gramian matrix.
Thus, for given accuracy, the complexity of this operation is $\mc O (n)$.
Second, Algorithm~\ref{alg: min sm sp abs} requires to determine solutions
to the pair of Lyapunov equations  \eqref{eq: lyapunov P and Q}. 
Modern methods to solve Lyapunov equations 
\cite{bartels1972solution} rely on the Schur decomposition of the 
matrix $\subscr{A}{av} $,  whose complexity is $\mc O(n^3)$.
It is worth noting that, given the Schur decomposition 
$\subscr{A}{av} = U T U^\tsp$, where $T$ is upper triangular and 
$U$ is unitary, a decomposition for $(\subscr{A}{av} - \tilde \alpha I)$
follows immediately by shifting $T$ to $(T-\tilde \alpha I)$. Therefore, a single 
decomposition is required at each iteration of the gradient descent and the 
complexity of  Algorithm~\ref{alg: min sm sp abs} is therefore
$\mc O(n^3)$.
  \oprocend
\end{remark}

\section{Distributed Gradient Descent}
\label{sec:distributed}
The centralized computation of $\{ d_1^*, \dots, d_m^*, \epsilon^* \}$ assumes 
complete knowledge of the matrices $A_1, \dots A_m$, and requires to 
numerically solve the pair of Lyapunov equations \eqref{eq: lyapunov P and Q}.
For large-scale traffic networks, such computation imposes a limitation in the 
dimension of the matrix $\subscr{A}{av}$ and, consequently, on the number of 
signalized intersections that can be optimized simultaneously.
Since the performance of the proposed optimization technique depends upon
the possibility of modeling and optimizing large network interconnections, 
a limitation on the number of modeled roads and intersections constitutes a 
bottleneck toward the development of more efficient traffic infrastructures.
A possible solution to address this complexity issue is to distribute the 
computation  of the descent direction in Algorithm \ref{alg: min sm sp abs} 
among a group of agents, in a way that each agent is 
responsible for a subpart of the entire computation (e.g. see 
Fig.~\ref{fig: distributed grid}). 
Agents can represent geographically-distributed control centers or clusters in 
parallel computing, each responsible for the control of a group of intersections.

In order to distribute the computation of solutions to 
\eqref{eq: lyapunov P and Q}, in this section we focus on distributively solving 
Lyapunov equations of the form
\begin{align}
\label{eq: lyapunov X}
\Lambda X + X \Lambda^\tsp + D = 0,
\end{align}
where $X = X^\tsp \in \real^{n \times n}$ is an unknown matrix, 
$D=D^\tsp \in \real^{n \times n}$, 
$\Lambda \in \real^{n \times n}$.
Let $\Lambda$ be partitioned as 
\begin{align*}
\Lambda = \Lambda_1 + \dots + \Lambda_\nu,
\end{align*}
where $\Lambda_i \in \real^{n \times n}$, $i \in \until \nu$. 
We assume that each agent $i$ knows $\Lambda_i$ only.
Note that $\Lambda_i$ are sparse matrices, and their sparsity pattern depends 
upon the subpart of infrastructure associated with that agent.
In addition, we assume that  neighboring agents are allowed to exchange 
information by means of a communication infrastructure.
Let $\mc G = (\mc V, \mc E)$ be the communication  graph, where each vertex 
$i \in \until \nu$ represents one agent, and $\mc E \subseteq \mc V \times \mc V$ 
represents the interconnection.
The following lemma constitutes the key ingredient of the method we propose 
to distributively compute $X$.

\begin{lemma}{\bf \textit{(Distributed Solutions to \eqref{eq: lyapunov X})}}
\label{lemma: distributed Lyapunov solution}
Consider a Lyapunov equation of the form \eqref{eq: lyapunov X}, and let 
$\Lambda$ be Hurwitz-stable.
The following statements are equivalent:
\begin{enumerate}[label=(\roman*)]
\item $X^*$ solves \eqref{eq: lyapunov X};
\item For all $i \in \until \nu$, there exists $D_i \in \real^{n \times n}$ s.t.
\begin{align*}
\Lambda_i  X^* + X^* \Lambda_i^\tsp + D_i = 0, \text{ and }
\sum_{i=1}^\nu D_i = D.
\end{align*}
\end{enumerate}
\end{lemma}
\smallskip
\begin{proof}
In order to prove the claim, we first observe that under the assumption of 
Hurwitz-stability for $\Lambda$,  the solution $X^*$ to 
\eqref{eq: lyapunov X}  is unique.\\
$(i) \Rightarrow (ii).$
Let $X^*$ denote the unique solution to \eqref{eq: lyapunov X}. 
By expanding 
$\Lambda = \Lambda_1 + \dots + \Lambda_\nu$, we obtain
\begin{align*}
\sum_{i=1}^{\nu}( \Lambda_i X^* + X^* A_i^\tsp )+ D =0.
\end{align*}
Thus, by letting  $D_i = -( \Lambda_i X^* + X^* A_i^\tsp )$, 
$(ii)$ immediately follows.\\
$(ii) \Rightarrow (i).$ 
Let $(\tilde X$, $\tilde D_1, \dots , \tilde D_\nu)$ satisfy $(ii)$, that is, 
for all $i \in \until \nu$
\begin{align*}
& \Lambda_i  \tilde X + \tilde X \Lambda_i^\tsp + \tilde D_i = 0, &&
& \sum_{i=1}^\nu \tilde D_i = \tilde D.
\end{align*}
Notice that the existence of the solution to \eqref{eq: lyapunov X} guarantees 
the existence of  $(\tilde X$, $\tilde D_1, \dots , \tilde D_\nu)$.
By substitution, we obtain  
$- \sum_{i+1}^\nu( \Lambda_i  \tilde X + \tilde X \Lambda_i^\tsp ) = D$, or 
in other words, $\tilde X$ satisfies 
$ \Lambda  \tilde X + \tilde X \Lambda^\tsp +D=0$.
The uniqueness of the solution to \eqref{eq: lyapunov X} implies 
$\tilde X = X^*$ and concludes the proof.
\end{proof}

The method we propose is based on a simultaneous and cooperative reconstruction 
of the matrices $X^*$ and $D_i$, and is described next.
For all $i \in \until \nu$, we let  
$\bar \Lambda_i = \Lambda_i \otimes I +  I \otimes \Lambda_i$.
Then, Lemma \ref{lemma: distributed Lyapunov solution} allows us to  
restate \eqref{eq: lyapunov X} as a system of equations of the form
\begin{align}
\label{eq: distributed linear system}
\underbrace{
\begin{bmatrix}
\bar \Lambda_1 & I & 0 & \cdots & 0 \\
\bar \Lambda_2 & 0 & I & & \vdots\\
\vdots  & \vdots & \ddots & \ddots \\
0 & I & \cdots & I & I
\end{bmatrix}
}_{H}
\underbrace{
\begin{bmatrix}
\supscr{X}{v} \\  \supscr{D}{v}_1 \\ \vdots \\  \supscr{D}{v}_\nu
\end{bmatrix}
}_w
= 
\underbrace{
\begin{bmatrix}
0 \\ 0 \\ \vdots \\ \supscr{D}{v}
\end{bmatrix}
}_z,
\end{align}
where $ \supscr{X}{v}, \supscr{D}{v}_1, \dots , \supscr{D}{v}_\nu$ are 
now the unknown parameters.

\begin{figure}[!t]
  \centering \subfigure[]{
   \raisebox{4mm}{
    \includegraphics[width=.39\columnwidth]{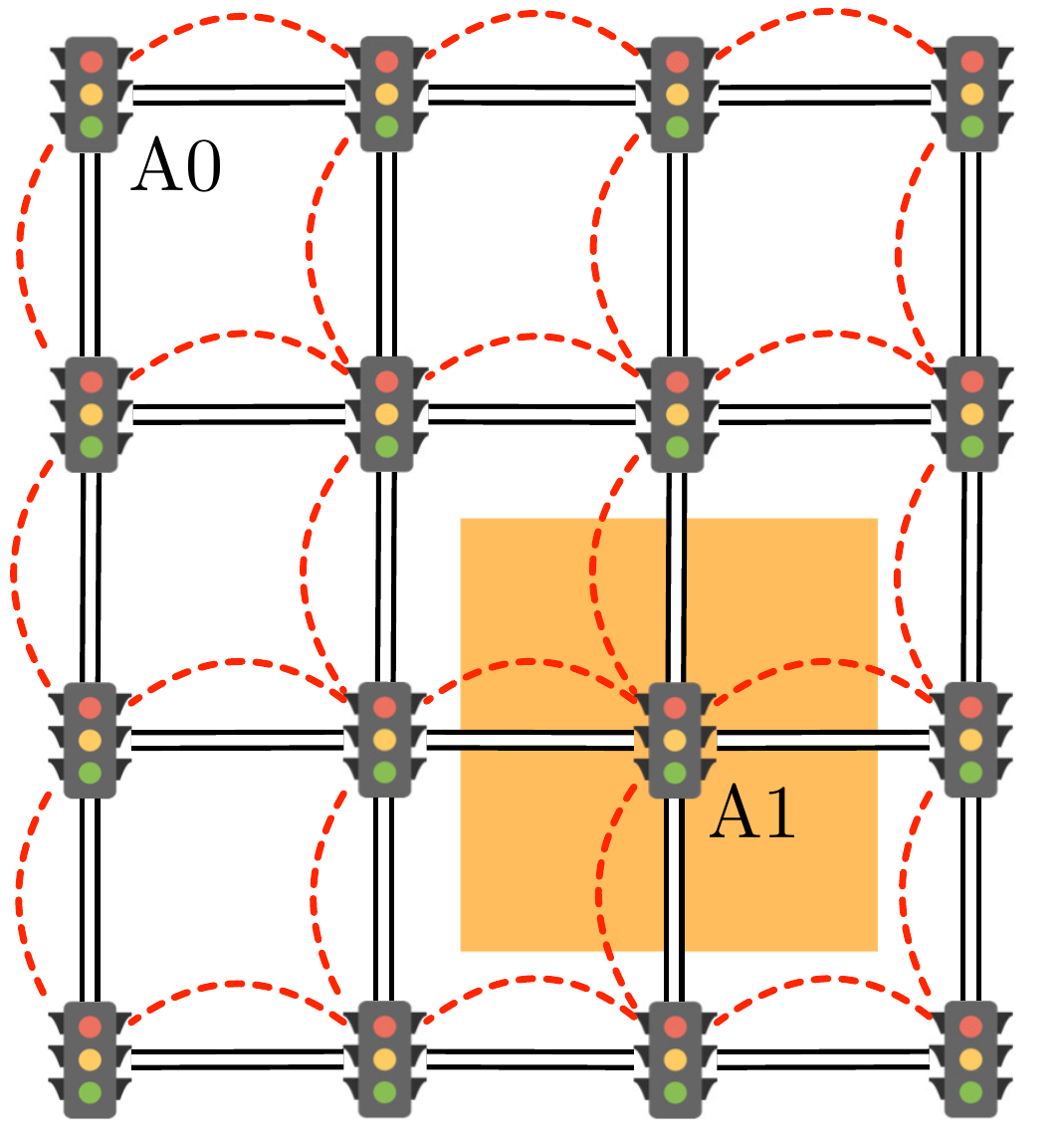} 
        \label{fig: distributed a} } }
 \subfigure[]{
    \includegraphics[width=.48\columnwidth]{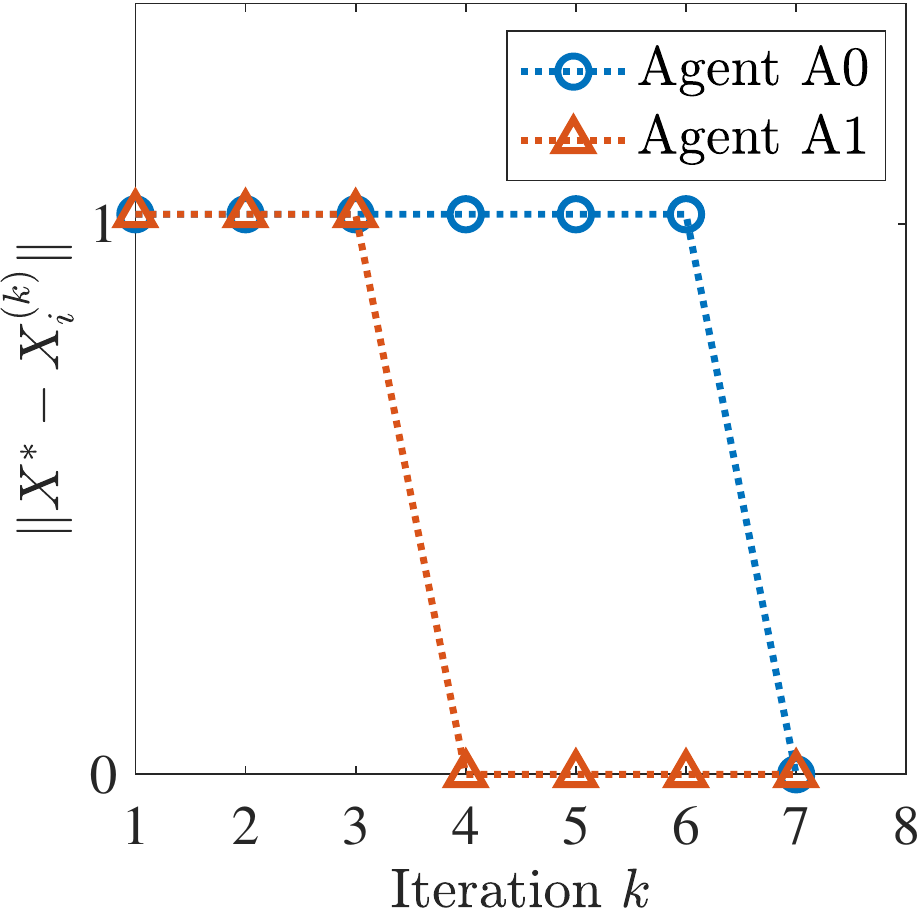}  
    \label{fig: distributed b}}
\caption[]{Manhattan-like traffic network and convergence of the 
distributed algorithm (i)-(iv). 
(a) Each agent, represented by a signalized intersection, has local knowledge 
of the road interconnection (continuous black lines) and is allowed to 
communicate with neighboring intersections (dashed red lines).
Shaded area depicts the subset information that is available to agent A1.
(b) Error between distributed solution and centralized solution vs iterations. 
The plot illustrates how internal agents (e.g. A1) show faster convergence due 
to shorter longest paths within the graph. }
  \label{fig: distributed grid}
\end{figure}

In order to distribute the computation of the vector $w$ (and thus $X^*$)
among the $\nu$  distributed agents, we let
\begin{align*}
H_i = 
\begin{bmatrix}
\bar \Lambda_i & 0 & \cdots & I & \cdots & 0 \\
0 & I & \cdots  & & \cdots & I 
\end{bmatrix},
\quad \quad
z_i = 
\begin{bmatrix}
0 \\ \supscr{D}{v}
\end{bmatrix},
\end{align*}
for all $i \in \until \nu$. 
At every iteration $k$, each agent $i$ constructs a local estimate  
$\hat w^{(k+1)}_i $ (correspondingly $\hat X^{(k)}_i$) by performing the following 
operations in  order:
\begin{enumerate}
\item[(i)] Receive $\hat w^{(k)}_j$ and $K_j^{(k)}$ from neighbor $j$;
\item[(ii)] $\hat w^{(k+1)}_i = \hat w^{(k)}_i +
 [K_i^{(k)} ~0] [K_i^{(k)} ~K_j^{(k)}]^\dagger  (\hat w^{(k)}_i - \hat w^{(k)}_j)$;
\item[(iii)] $K_i^{(k+1)} = \Basis ( \text{Im}(K_i^{(k)}) \cap \text{Im}(K_i^{(k)}) )$;
\item[(iv)] Transmit $\hat w^{(k+1)}_i$ and $K_i^{(k+1)}$ to neighbor $j$;
\end{enumerate}
where,
\begin{align*}
\hat w_i^{(0)} = H_i^\dagger  z_i,
\quad \quad
K_i^{(0)} = \Basis( \Ker(H_i) ).
\end{align*}

While the convergence of the procedure (i)-(iv) can be ensured by means of an
approach similar to \cite{FP-RC-FB:10p}, in this paper we focus on providing 
numerical  validation to the algorithm, in the framework of transportation systems.
To this aim, we consider the Manhattan-like network interconnection 
depicted in Fig.~\ref{fig: distributed grid} \cite{lammer2008self}, 
and assume that each traffic intersection is equipped with a computational unit 
that is responsible for a  subpart of the computation of 
\eqref{eq: lyapunov P and Q}.
In other words, every traffic intersection denotes one agent of the communication 
graph $\mc G$, and is allowed to exchange information 
with neighboring intersections by means of  communication channels 
(dashed-red lines in Fig.~\ref{fig: distributed grid}).
We make the practical assumption that each agent has the sole knowledge of: 
(i) the local structure of the traffic interconnection, that is, the geographical layout 
of roads adjacent to that intersections (shaded area in Fig.~\ref{fig: distributed a}), 
and 
(ii) the current travel time of roads adjacent to that intersection.

We numerically illustrate in Fig.\ref{fig: distributed b} the convergence of the 
distributed procedure (i)-(iv), where we compare the accuracy of the 
local estimate $\hat X^{(k)}_i$ with respect to the centralized solution $X^*$, as 
a function of the iteration $k$.
As discussed in \cite{FP-RC-FB:10p}, this class of procedures shall compute
$\hat X^{(k)}_i = X^*$ in at most $\text{diam}(\mc G)$ steps, where 
$\text{diam}(\mc G)$ denotes the diameter of the communication graph 
$\mc G$.

\begin{figure}[!b]
  \centering \subfigure[]{
    \includegraphics[width=.75\columnwidth]{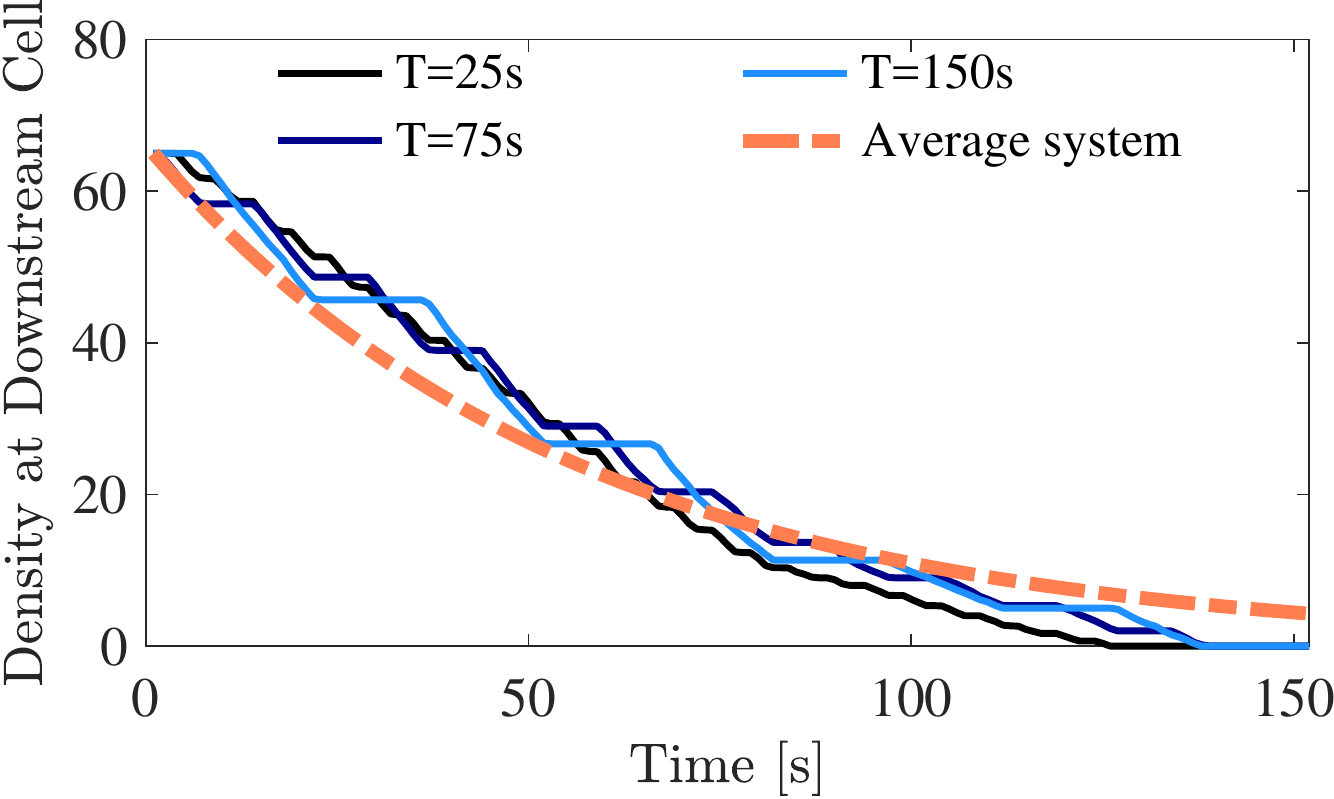} 
    \label{fig:averageVsSwitching_evolution}  }
 \subfigure[]{
    \includegraphics[width=.75\columnwidth]{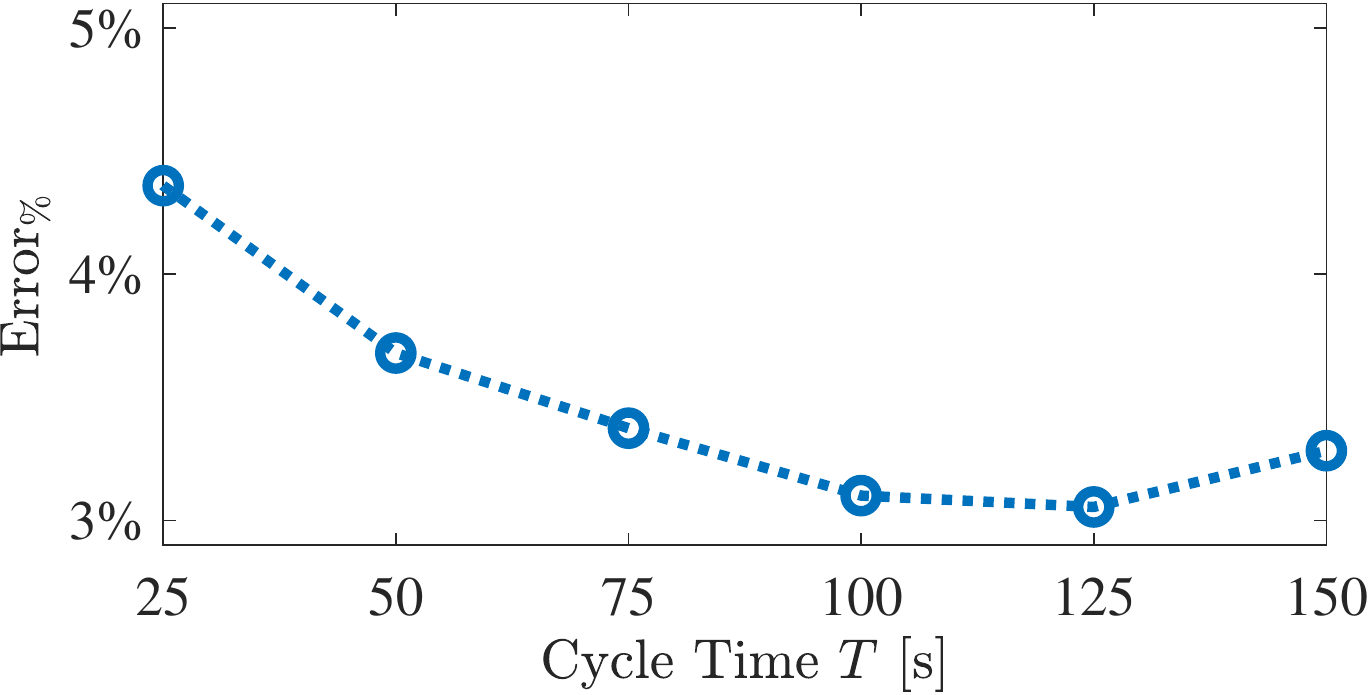}  
    \label{fig:averageVsSwitching_error}}
  \caption[]{Accuracy of average dynamical model \eqref{eq: average system} 
  with respect to microscopic simulations for a single signalized road. 
  (a) Time evolution of the density at the downstream cell for different 
  intersection cycle time $T$.   
  (b) Inaccuracies of the average dynamics in relation to the intersection cycle 
  time $T$.}
  \label{fig:averageVsSwitching}
\end{figure}

\begin{figure}[t]
  \centering \subfigure[]{
    \includegraphics[width=.44\columnwidth]{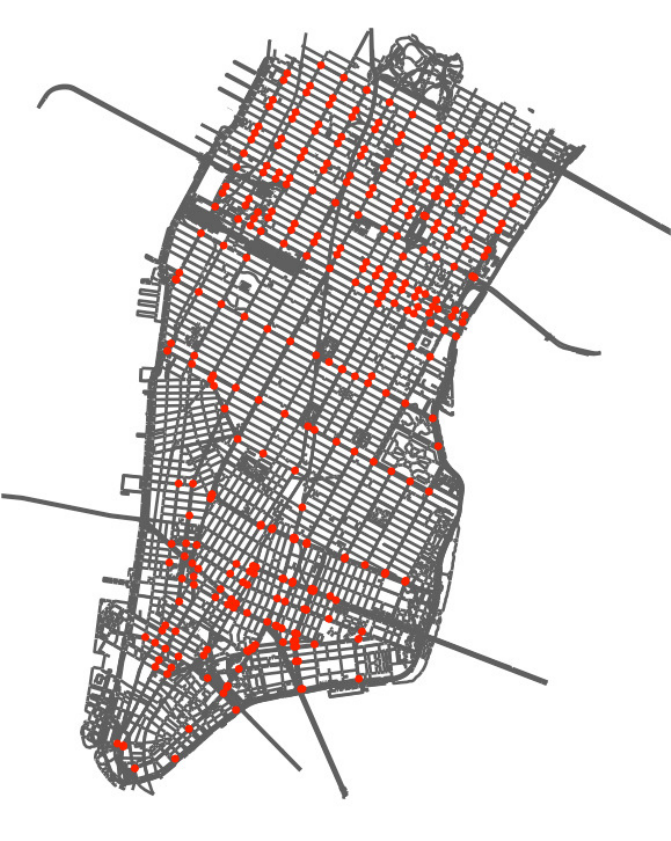} 
    \label{fig:manhattan_lights}  }
 \subfigure[]{
    \includegraphics[width=.46\columnwidth]{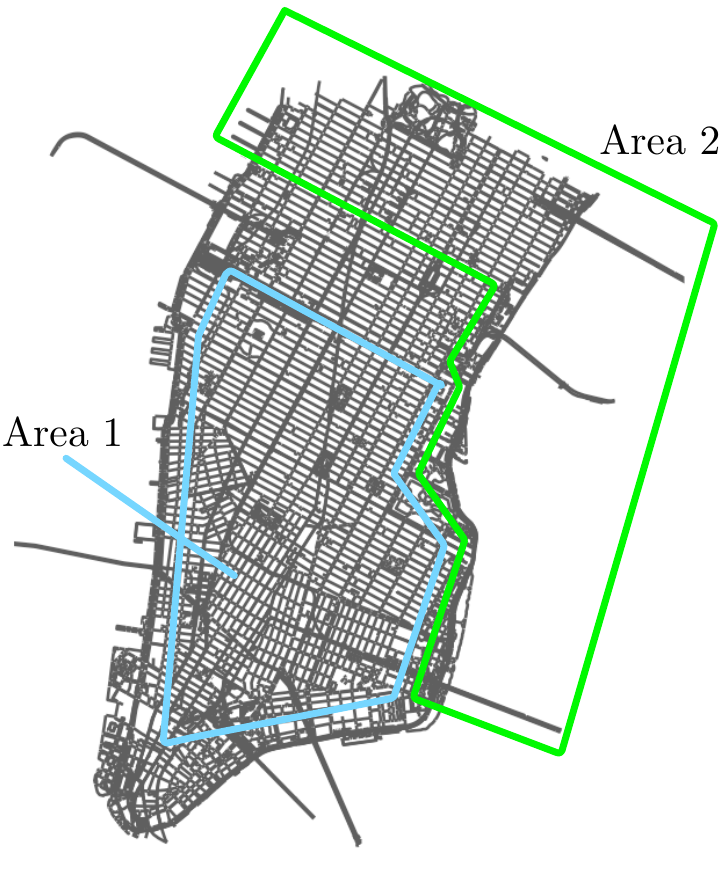}  
    \label{fig:manhattan_lights_areas}}
  \caption[]{Geographic map of urban traffic interconnection in Manhattan. 
  (a) Signalized intersections considered (red dots).   
  (b) Commute areas considered.}
  \label{fig:manhattan_map}
\end{figure}

\begin{figure*}[!t]
 \subfigure[]{
    \includegraphics[width=.65\columnwidth]{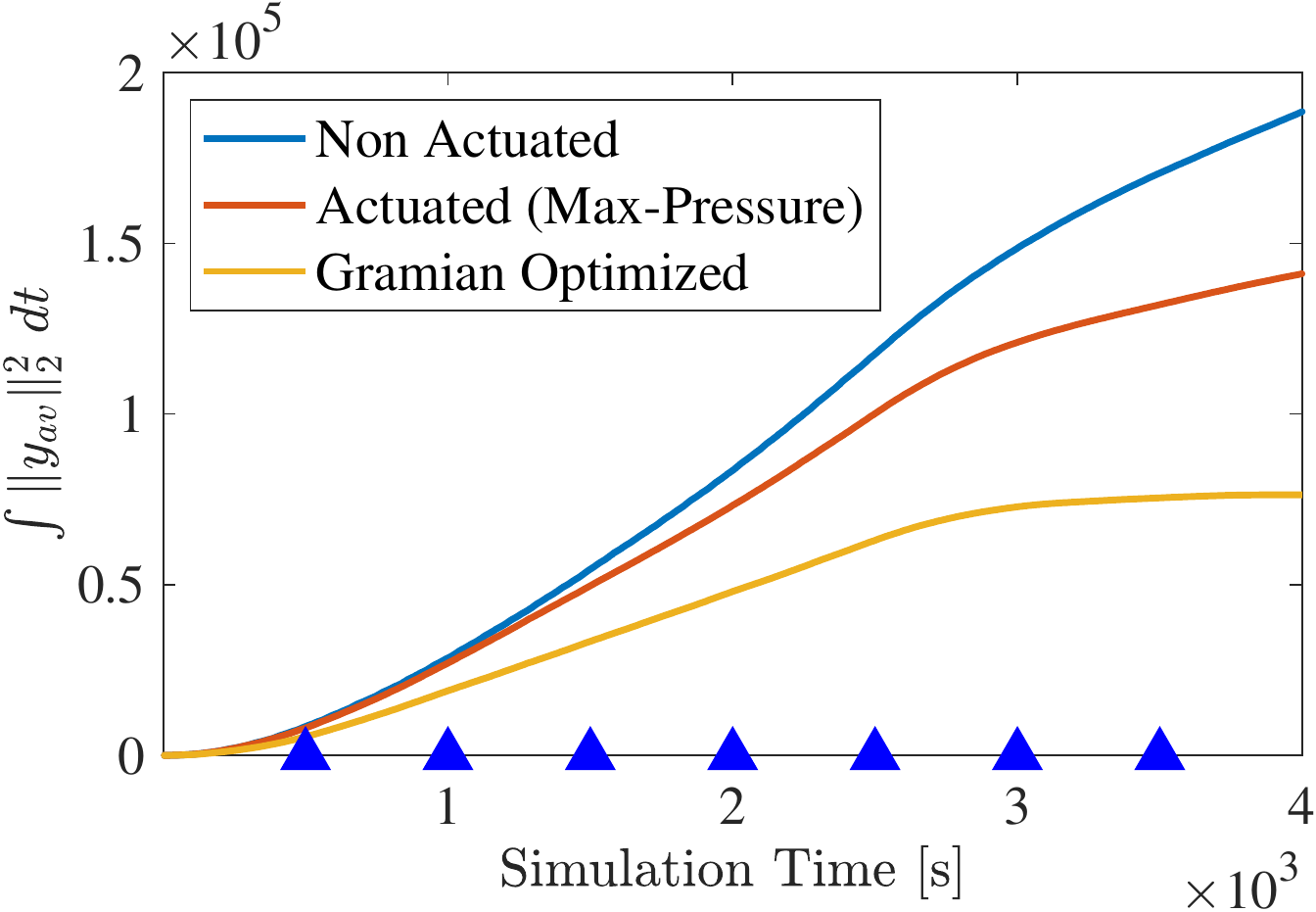}  
    \label{fig:manhattan_lights-costFcn}}
  \centering \subfigure[]{
    \includegraphics[width=.65\columnwidth]{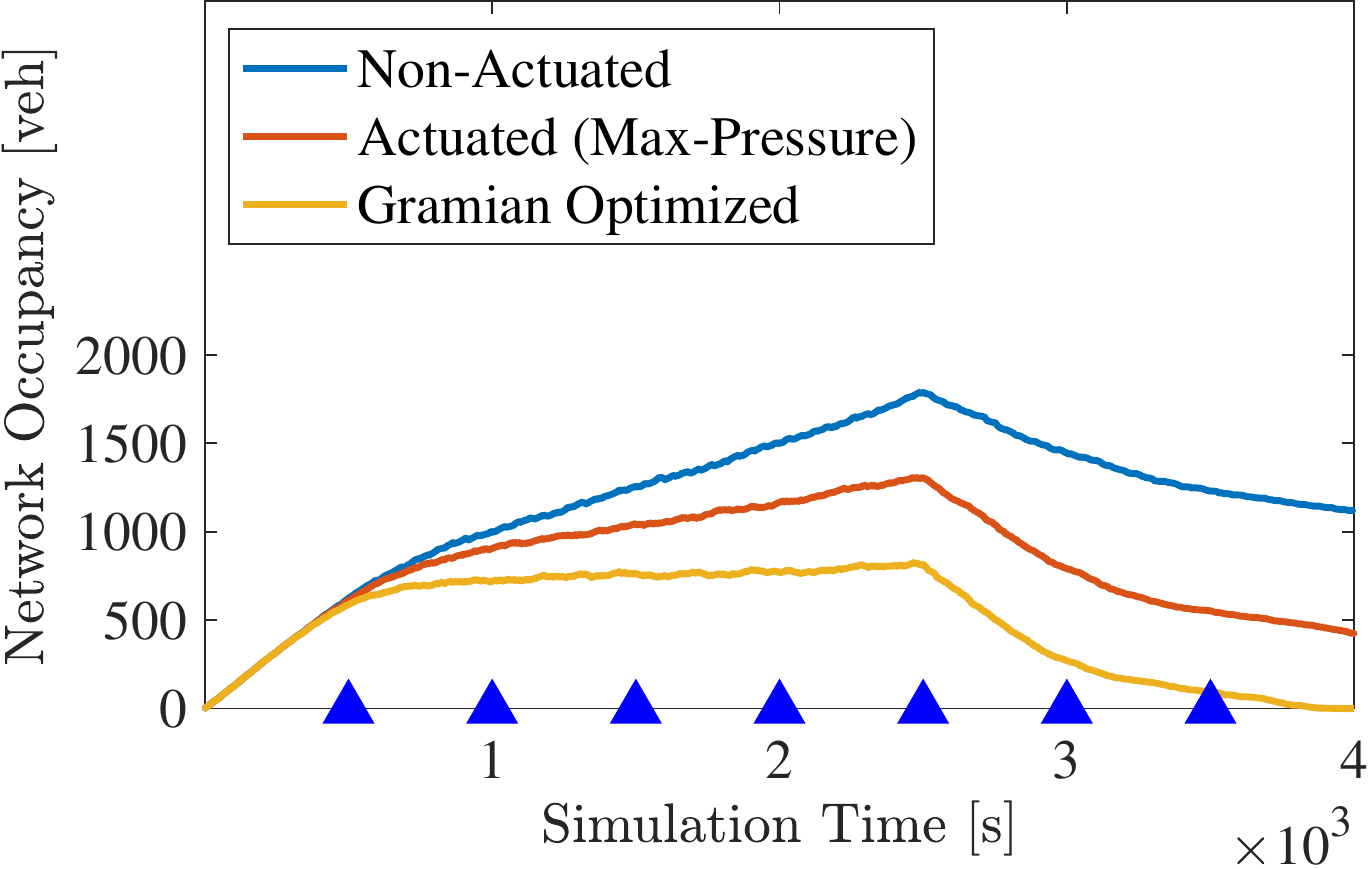} 
    \label{fig:SPABvsBP_noOfVehicles}  }
  \caption[]{Network performance of the Manhattan urban interconnection 
  assessed via a microscopic simulation for three control policies. 
Blue triangles denote the time instants where the solution to
\eqref{opt: epsilon cost function} is re-computed with updated $x_0$.
(a) Cost of congestion associated with traffic signaling 
(see optimization \eqref{opt: informal}). 
(b) Total number of vehicles in the network.}
  \label{fig:SPABvsBP}
\end{figure*}

\section{Simulations}
\label{sec: simulations}
This section provides numerical simulations in support to the 
assumptions  made in this paper,  and includes discussions and 
demonstrations of  the benefits of our methods.
We generate test cases using real-world traffic networks from the 
OpenStreet Map database and validate the techniques on a 
simulator based  on \emph{Sumo} 
(\emph{Simulation of Urban MObility} \cite{behrisch2011sumo}).

We first consider a single road, connected at downstream to a signalized 
intersection (e.g. Fig.~\ref{fig: phases}), and illustrate in 
Fig.~ \ref{fig:averageVsSwitching_evolution} a comparison between the 
trajectories of the switching system \eqref{eq: switching model} and 
the trajectories of the average dynamical model \eqref{eq: average system}.
The figure shows the discharging pattern over time of the cell at road 
downstream, resulting from a microscopic simulation.
In the simulation, the initial downstream cell density is $x^\sigma = 65$veh
and, for simplicity, there are no inflows to the road. 
Fig.~\ref{fig:averageVsSwitching_error} demonstrates the accuracy of  
the approximation for different choices of the split cycle time $T$.
The precision of the model can be quantitatively assessed by considering the 
following error-measure over a certain time horizon $[ 0, \mc H]$:
\begin{align}
\label{eq:errorPercent}
\text{Error}_{\%} = \frac{1}{\mc H } \int_0^{\mc H}  
	\frac{\| x - \subscr{x}{av} \|}{\|\subscr{x}{av}\|} ~ dt.
\end{align}
The relative error measure \eqref{eq:errorPercent} captures the deviation of 
\eqref{eq: average system} from the microscopic simulation, normalized with 
respect to the road density and the considered time horizon.
As indicated in Fig.~\ref{fig:averageVsSwitching_error}, the inaccuracies
due to linearization and time-averaging are under $5\%$ for common 
intersection cycle times.

\begin{table}[!t]
\caption{Manhattan Network Inflow Rates}
\begin{center}
\begin{tabular}{|c|c|c|}
\hline
\textbf{Time [sec]}&&\\
\textbf{(From - To)} & \textbf{\textit{Area 1 Inflow [veh/h]}} & \textbf{\textit{Area 2 Inflow [veh/h]}} \\
\hline
$0 - 2500$ & $4000$ & $0$  \\
$2500 - 4000$ & $0$ & $0$  \\
\hline
\end{tabular}
\label{tab:areaInflowRates}
\end{center}
\end{table}

\begin{figure*}[t]
  \centering \subfigure[]{
    \includegraphics[width=.26\columnwidth]{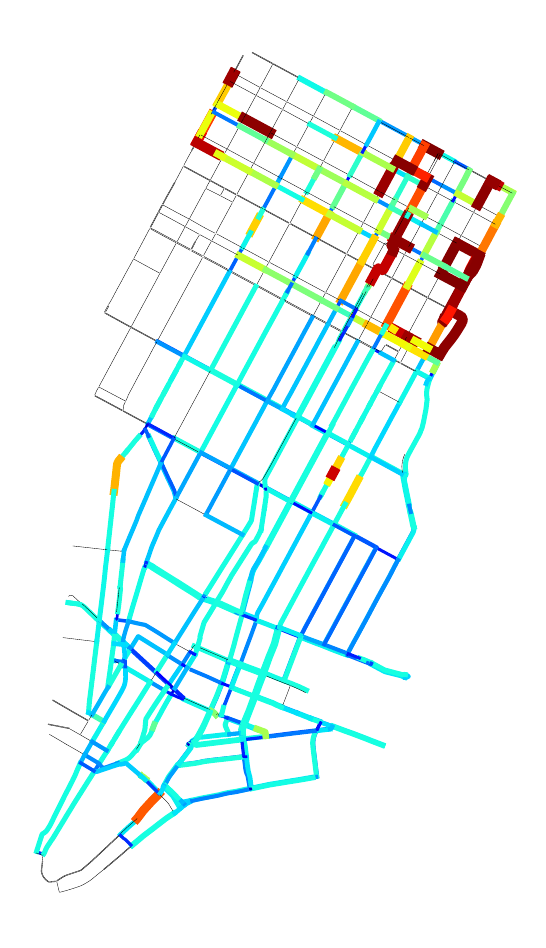} }
    \subfigure[]{
    \includegraphics[width=.26\columnwidth]{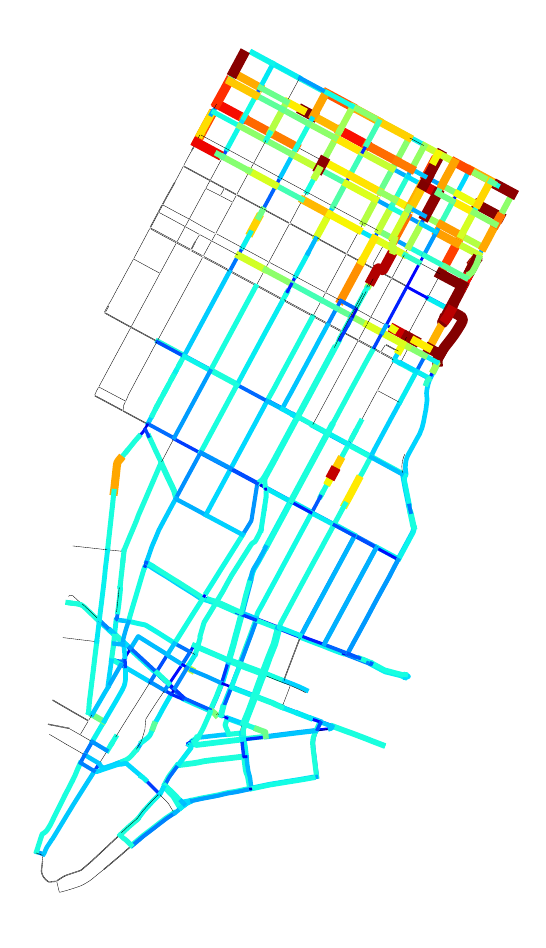} }
    \subfigure[]{
    \includegraphics[width=.26\columnwidth]{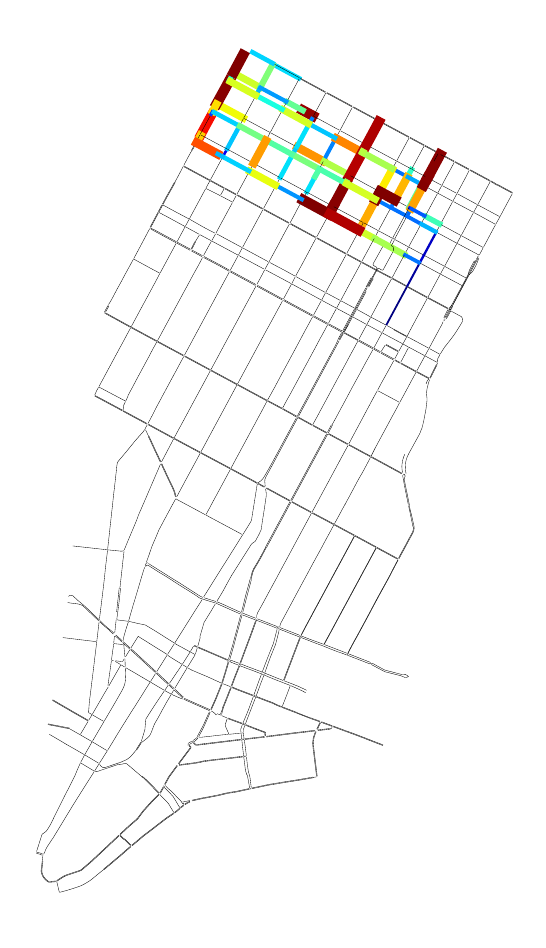} }
    \subfigure[]{
    \includegraphics[width=.33\columnwidth]{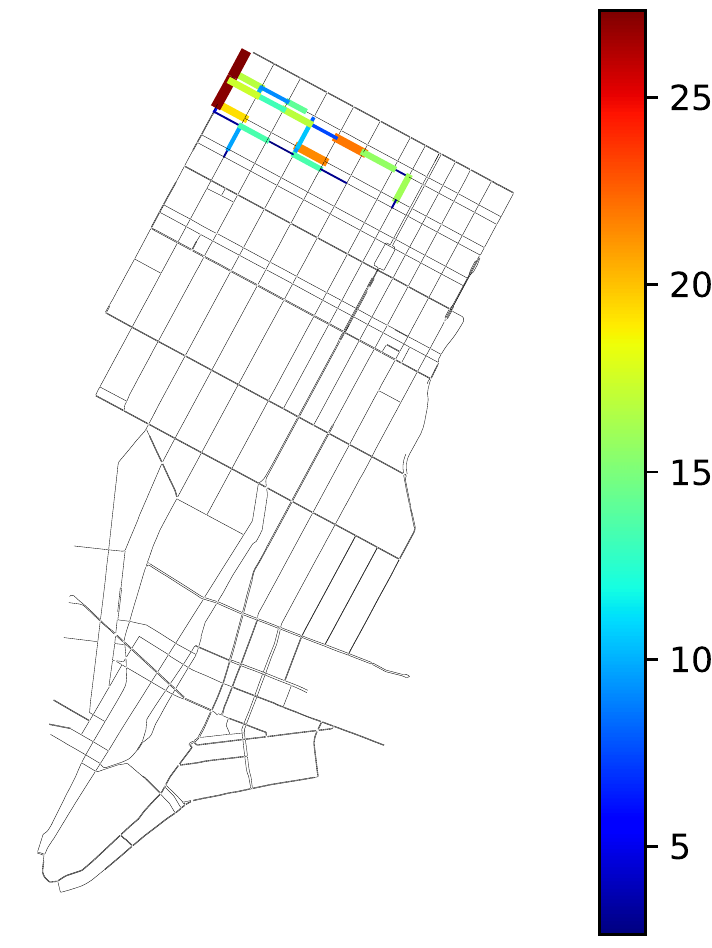}}\\
    \vspace{-.1cm}
 \subfigure[]{
    \includegraphics[width=.26\columnwidth]{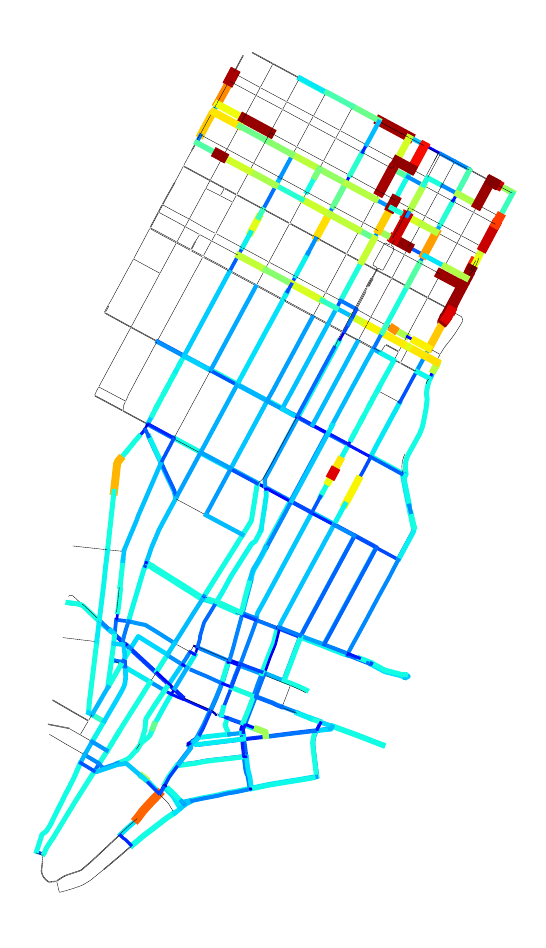} }
    \subfigure[]{
    \includegraphics[width=.26\columnwidth]{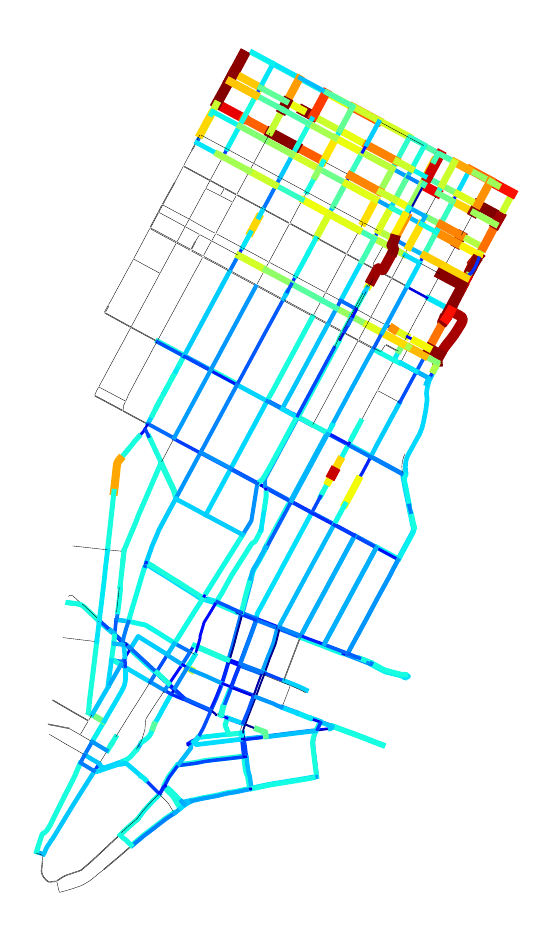} }
    \subfigure[]{
    \includegraphics[width=.26\columnwidth]{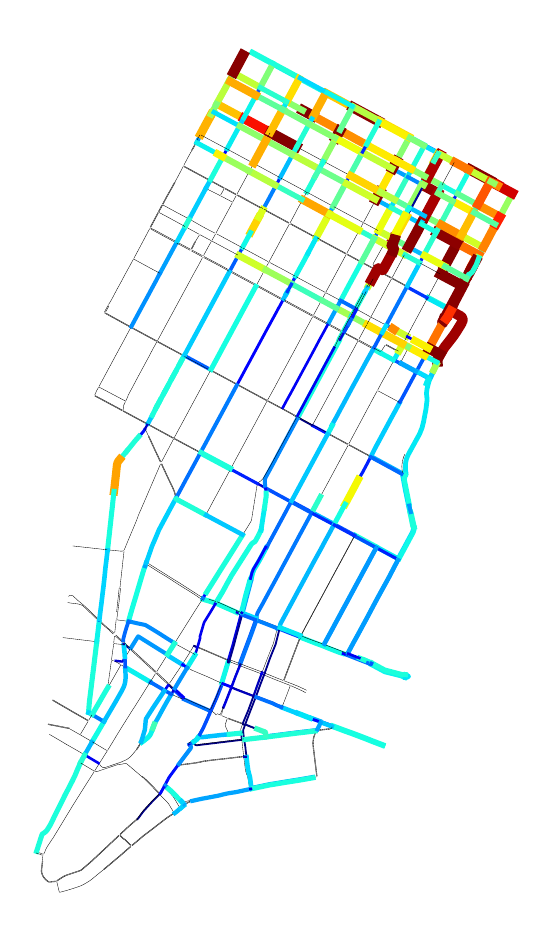} 
    \label{fig:manhattan}  }
    \subfigure[]{
    \includegraphics[width=.33\columnwidth]{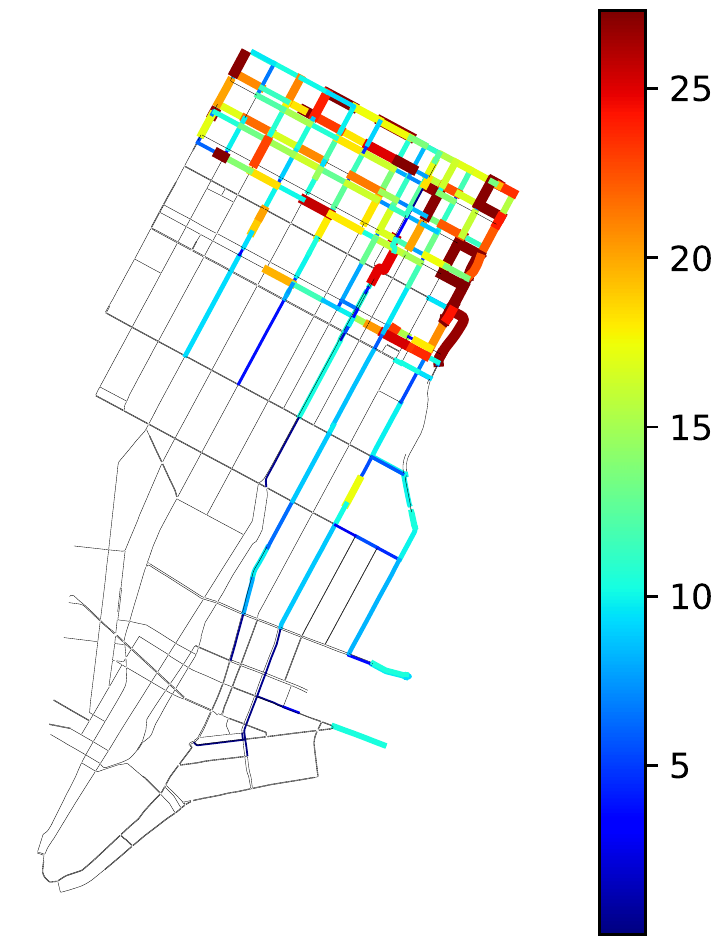} }
  \caption[]{Time evolution of the network state (traffic density [veh/cell]) for:
(a)-(d)   solution to \eqref{opt: informal};
(e)-(h)  Max-Pressure policies. 
((a) and (e)) $\text{simulation time}=1000$;
((b) and (f)) $\text{simulation time}=2000$; 
((c) and (g)) $\text{simulation time}=3000$; 
((d) and (h)) $\text{simulation time}=4000$.
The comparison shows that the solution to \eqref{opt: informal} leads to 
faster traffic dynamics that results in reduced travel time.}
  \label{fig:manhattan_time_evolution}
\end{figure*}

We then consider a test case (Fig.~\ref{fig:manhattan_map}) 
inspired by the area of Manhattan, NY,
which features $\subscr{n}{r}=958$ 
roads and $n_{\mc I} = 332$ signalized intersections. 
We replicate a daily-commute scenario, where sources of traffic flows
$\mc S$  are uniformly distributed in the  central area of the island (Area 1), 
and routing is chosen so that traffic flows are departing from the city, that is,
destinations $\mc D$ are uniformly distributed within Area 2. 
Each vehicle follows the shortest path between its source-destination pair, and the discretization is performed with $h=0.1$ miles.
The inflow-rates employed in the simulation are summarized in 
Table~\ref{tab:areaInflowRates}.
Fig.~\ref{fig:manhattan_lights-costFcn} shows a comparison, in terms of 
the cost function in \eqref{opt: informal}, resulting from the  microscopic 
simulation of:
(i) an implementation of the optimal solution to 
\eqref{opt: epsilon cost function},
(ii) adaptive distributed control policies (Max-Pressure 
\cite{varaiya2013max}), and
(iii) fixed time control \cite{papageorgiou2003review}.
For scenario (i), the output of the optimization 
\eqref{opt: sp abs cost fcn, fixed epsilon}
is employed to construct feasible sets of phases at each intersection, where 
cycle time is set to $T=100$sec, 
and no optimization is performed to decide the specific sequence of  
phases at each intersection.
Consequently to the argument presented in 
Section~\ref{subsec:problemFormulation},
the solution to \eqref{opt: epsilon cost function} is re-computed with  
updated traffic conditions (i.e. $x_0$) every $500$sec.
Fig.~\ref{fig:SPABvsBP_noOfVehicles} shows a comparison of the total number 
of vehicles in the network, for the three control techniques considered.
As highlighted from the simulation, the availability of an overall network model,
capable of capturing all the relevant network dynamics and dynamical 
interactions between diverse network components, results in reduced roads 
occupancy, which in turn implies improved network congestion, throughput, 
and vehicle travel time \cite{gomes2006optimal}. 
The effects of different control policies on the network overall congestion can 
be further visualized by means of the graphical illustration in 
Fig.~\ref{fig:manhattan_time_evolution}.
The figure shows the geographical displacement of traffic over time.
The graphic highlights that the the possibility to estimate the network current 
traffic conditions, combined with optimization problems of the form
\eqref{opt: informal}, result in faster system responses against congestion, 
that ultimately result in shorter ''cool down`` phases.
Notably, the network is evacuated faster as compared to 
techniques that rely on local knowledge of the traffic conditions.
Finally, the above simulation (Fig.~\ref{fig:manhattan_time_evolution}) also 
validates the argument presented in Section~\ref{subsec:problemFormulation}, 
that is, when exogenous inflows are unknown, the current inflow demands 
enter the optimization through the initial state $x_0$, and motivates our 
formulation \eqref{opt: informal}.

\section{Conclusions}
\label{sec: conclusions}
This paper describes a simplified model to capture in a tractable way the 
overall dynamics of urban traffic networks.
This model allows us to reformulate the goal of optimizing network congestion 
as the problem of minimizing  a metric of controllability of an appropriately 
defined dynamical system  associated with the network.
Our results show that the availability of an approximate model of the overall 
network can considerably improve the network efficiency, and allows for a 
more tractable analysis compared to traditional models.
We show how the performance of other distributed policies deteriorate  as 
compared to centralized models, due to the lack of availability of a global 
dynamical model capable of capturing all the relevant network dynamics.
We adapt our technique to fit distributed and/or parallel implementations, and 
provide an efficient way to optimize large groups of intersections under the 
practical assumption that each agent has only local knowledge on the 
infrastructure.
%
We envision that our model of traffic network and the proposed optimization 
framework will be useful in future research targeting design of traffic networks, 
control, and security analysis.


%

%
%
%

\ifCLASSOPTIONcaptionsoff
  \newpage
\fi


\bibliographystyle{elsarticle-num} 
\bibliography{alias,FP,BIB_traffic}

%

%
%




\end{document}